\documentclass[12pt]{article}
\usepackage[utf8]{inputenc}
\usepackage[margin=1in]{geometry}
\usepackage{graphicx}
\usepackage{booktabs}
\usepackage{array}
\usepackage{paralist}
\usepackage{verbatim}

\usepackage[nottoc,notlot,notlof]{tocbibind}

\usepackage{amsfonts}
\usepackage{amsmath}
\usepackage{amssymb}
\usepackage{amsthm}
\usepackage{mathrsfs}
\usepackage{mathtools}
\usepackage{stmaryrd}
\usepackage{esint}
\usepackage{bbm}
\usepackage[dvipsnames]{xcolor}
\usepackage{float}
\usepackage{xfrac}
\usepackage{nicefrac}

\usepackage{newtxtext}
\usepackage{newtxmath}
\usepackage{anyfontsize}

\usepackage{enumerate}
\usepackage{enumitem}
\usepackage{longtable}

\usepackage[colorlinks=true, pdfstartview=FitV, linkcolor=blue, citecolor=blue, urlcolor=blue]{hyperref}
\usepackage[normalem]{ulem}\usepackage{cleveref}

\usepackage{tikz}
\usetikzlibrary{calc}
\usepackage{pgf}
\usetikzlibrary{external}
\numberwithin{equation}{section}
\numberwithin{figure}{section}

\newtheorem{theorem}{Theorem}[section]
\newtheorem*{theorem*}{Theorem}
\newtheorem{lemma}[theorem]{Lemma}
\newtheorem{proposition}[theorem]{Proposition}
\newtheorem{corollary}[theorem]{Corollary}
\theoremstyle{definition}

\newtheorem{remark}[theorem]{Remark}

\newcommand*{\NN}{\ensuremath{\mathbb{N}}}

\newcommand*{\TT}{\ensuremath{\mathbb{T}}}
\newcommand*{\ZZ}{\ensuremath{\mathbb{Z}}}
\newcommand*{\RR}{\ensuremath{\mathbb{R}}}
\newcommand*{\CC}{\ensuremath{\mathbb{C}}}

\newcommand{\eps}{\varepsilon}

\DeclarePairedDelimiter{\norm}{\lVert}{\rVert}

\let\tmpsum\sum
\renewcommand{\sum}{\tmpsum\nolimits}

\usepackage{titlesec}
\newcommand{\addperiod}[1]{#1.}
\titleformat{\section}{\centering\normalfont\Large}{\thesection.}{0.5em}{}
\titleformat*{\subsection}{\bfseries}
\titleformat{\subsubsection}[runin]{\normalfont\bfseries}{\thesubsubsection.}{0.5em}{\addperiod}
\titleformat*{\subsubsection}{\normalfont\itshape}
\titleformat*{\paragraph}{\bfseries}
\titleformat*{\subparagraph}{\large\bfseries}

\title{Multi-Spectral QPAT with Frequency-Averaged Measurements}

\author{
Yunhao Sun\thanks{Center for Applied Mathematics, Cornell University.
{\footnotesize \href{mailto: ys2337@cornell.edu}{ys2337@cornell.edu}.}
}\qquad
Yang Yang\thanks{Department of Computational Mathematics, Science and Engineering, Michigan State University.
{\footnotesize \href{mailto: yangy5@msu.edu}{yangy5@msu.edu}.}
}\qquad
Yunan Yang\thanks{Department of Mathematics, Cornell University.
{\footnotesize \href{mailto: yunan.yang@cornell.edu}{yunan.yang@cornell.edu}.}
}
}

\date{}
\begin{document}

    \maketitle

   \begin{abstract}
        We study some inverse problems in a frequency-averaged model for quantitative photoacoustic tomography, in which the optical parameters are modeled by a finite number of spectral functions and spatial coefficients, and the measurements are given by internal observables averaged over a spectral bandwidth. Using multiple measurements generated by suitable complex geometrical optics solutions, we establish uniqueness and Lipschitz-type stability for the associated linearized inverse problem for diffusion and absorption recovery. This linearized stability further yields local H\"older-type stability for the nonlinear inverse problem via an abstract linear-to-nonlinear argument.
    \end{abstract}

    \section{Introduction}

    Photoacoustic tomography (PAT) is a hybrid modality that combines the sensitivity of optical imaging to tissue composition with the spatial resolution of ultrasound. In PAT, a short near-infrared light pulse is delivered to an optically heterogeneous medium, where absorbed energy produces transient heating and thermoelastic expansion. The resulting pressure field propagates as ultrasound and is measured at the boundary. Because acoustic scattering in tissue is substantially weaker than optical scattering, PAT can image optical contrast at higher spatial resolution than standalone optical imaging. We refer to~\cite{Beard2011,Kuchment2012Hybrid,KuchmentKunyansky2011,Scherzer2011,Wang2004,Wang2008} and the references therein for overviews of its physical principles and applications.

    PAT is typically formulated as two coupled inverse problems. The acoustic stage recovers the absorbed-energy distribution from boundary ultrasound measurements as an inverse source problem for the wave equation. The optical stage, known as \textit{quantitative PAT (QPAT)}, then recovers tissue optical parameters from this distribution. In the diffusive regime of light propagation, that is, when scattering dominates absorption and propagation distances exceed the transport mean free path, light propagation can be well approximated by a spatial diffusion equation. The resulting model is known as \textit{diffusive QPAT}. 
    We remark that PAT models other than the two-stage model have also been proposed in the literature; see, for instance,~\cite{Ghandriche2023,Pan2025,Haltmeier2015,Ding2015,Frederick2018,Zhao2021}. In this paper, we adopt the two-stage model and focus on diffusive QPAT. 
    
    The inverse problem in diffusive QPAT aims to recover optical parameters from the reconstructed absorbed energy under prescribed illuminations. Conventional formulations typically idealize the incident light as monochromatic or sufficiently narrowband, so that light propagation and optical coefficients are modeled at a single frequency. 
    This approximation is appropriate when the source bandwidth is negligible compared with the frequency scale over which the tissue optical properties vary. 
    However, in practice, every light source has a nonzero spectral bandwidth, and illuminations may interact with frequency-dependent absorption and scattering across a nontrivial range of frequencies. If the resulting photoacoustic response is not spectrally resolved, the measured absorbed energy represents the combined contribution of these frequency components. Motivated by this setting, the present work considers a frequency-dependent diffusion model driven by polychromatic illumination and studies the recovery of the underlying spatial coefficient components from frequency-averaged internal measurements.

    \subsection{Frequency-Resolved QPAT}
    We begin with the conventional diffusive QPAT model, in which the light source is assumed to be monochromatic, so that light propagation and optical parameters are modeled at a single frequency. We refer to this model as the \textit{frequency-resolved diffusive QPAT model}. 

    Let $\Omega\subseteq \RR^n$ $(n \geq 3)$ be a bounded domain with smooth boundary, modeling the soft tissue. The spatial distribution of the radiance in the diffusive regime satisfies the following boundary value problem:
    \begin{equation}\label{QPAT}
        \left\{
        \begin{aligned}
            L u := -\nabla\cdot (D(x) \nabla u(x)) + A(x) u(x) &= 0 \quad && x\in \Omega,\\
            u(x) \vert_{\partial\Omega} &= f(x) \quad &&x\in \partial\Omega.
        \end{aligned}
        \right.
    \end{equation}
    Here, $D$ is the \textit{diffusion coefficient}, $A$ is the \textit{absorption coefficient}, $f(x)$ models the user-specified \textit{boundary illumination}. The operator $L$, defined by the left-hand side of the equation, is a second-order elliptic operator.
    The measurement, that is, the absorbed energy, is modeled as 
    $\mathfrak{m}(x) := A(x)u(x)$.
    (Strictly speaking, the energy absorption process in QPAT involves another optical parameter known as the \textit{Gr\"uneisen coefficient}. In this paper, we assume that the Gr\"uneisen coefficient is known and focus only on the pair of optical parameters $(D,A)$.)
    The inverse problem in frequency-resolved diffusive QPAT is to recover $D$ and $A$ from one or more internal measurements of the form $\mathfrak{m}(x)$, given one or more user-specified boundary illuminations $f$.

    \textbf{Literature Review:} The inverse problem for the frequency-resolved QPAT model has been studied extensively. In particular, uniqueness and stability for frequency-resolved diffusive QPAT have been established using well-chosen internal data sets~\cite{Bal2010}, a priori illuminations in three dimensions~\cite{Alessandrini2017}, layered-medium assumptions~\cite{RenTriki2019}, two point-source measurements~\cite{Bonnetier2022}, partial boundary illuminations~\cite{chen2012quantitative}, and partial internal data~\cite{TrikiXue2021}. Linearized results include semi-Fredholm stability~\cite{Kuchment2012}, stability from one internal measurement~\cite{Montalto2013}, and generic uniqueness for nonlocal data~\cite{Kuchment2015}.

    Concerning joint recovery of all the diffusion, absorption, and Gr\"uneisen coefficients, Bal and Ren showed that single-spectral data can recover at most two of them~\cite{Bal2011}. However, uniqueness for all three is available under a piecewise-constant assumption~\cite{Naetar2014}. Multi-spectral data can overcome the single-spectral obstruction under suitable spectral assumptions~\cite{BalRen2012Multispectral}. Related chromophore studies analyze wavelength-dependent fluence in spectral unmixing~\cite{Cox2009Chromophores} and nonlinear model-based estimation of absolute concentrations~\cite{Laufer2010Chromophores}.

    Numerical approaches have also been developed for QPAT imaging. These include iterative absorption recovery~\cite{Cox2006}, simultaneous absorption--scattering inversion using transport and diffusion models~\cite{Tarvainen2012}, hybrid optimization with internal and boundary optical data~\cite{RenGaoZhao2013}, topological-derivative reconstruction~\cite{Aspri2020}, direct inversion in realistic acoustic media~\cite{Javaherian2019}, and finite-element output-least-squares methods with error estimates~\cite{Alberti2026FEM}. Statistical approaches include Bayesian reconstruction~\cite{Tarvainen2013Bayesian}, spectral Bayesian inversion~\cite{Pulkkinen2014Bayesian}, direct estimation from acoustic time series~\cite{Pulkkinen2016Direct}, and adaptive Monte Carlo reconstruction~\cite{Hanninen2024}.

    Although the works above focus primarily on diffusive QPAT, related models incorporating more complex light physics have also been studied in the literature. These include QPAT governed by the radiative transport equation~\cite{Bal2010Transport,Ren2015Transport,Saratoon2013,Mamonov2014} and diffusion models with nonlinear absorption~\cite{Ren2018Nonlinear} and two-photon effects~\cite{Lai2022Nonlinear}.

    \subsection{Frequency-Averaged QPAT}
    The frequency-resolved model is appropriate only for highly narrowband illumination. In practice, light sources have nonzero spectral bandwidths, over which the diffusion and absorption coefficients vary. When the absorbed energy does not resolve optical frequency, the measurement consists of frequency-dependent contributions over the entire illumination band. To account for such effects, we introduce the following frequency-dependent diffusive QPAT model:
    \begin{equation}\label{frequency-QPAT}
        \left\{
        \begin{aligned}
            L_t u := -\nabla\cdot (D(x, t) \nabla u(x, t)) + A(x, t) u(x, t) &= 0 \quad &&(x,t)\in \Omega\times \TT^d,\\
            u(x, t) \vert_{\partial\Omega} &= f(x, t) \quad &&(x, t)\in \partial\Omega\times\TT^d.
        \end{aligned}
        \right.
    \end{equation}
    Here, $t$ is the frequency parameter, and $\TT^d := [0,1)^d$ is the $d$-dimensional torus representing the rescaled frequency band. In practice, frequency is one-dimensional, corresponding to $d=1$. We consider a general $d$-dimensional parameter so that the model can also describe dependence on other physical parameters (e.g., temperature), if necessary. All notation denotes the same quantities as in the frequency-resolved model, except that they are now $t$-dependent. We refer to this model as the \textit{frequency-averaged diffusive QPAT model}.

    The dependence of the optical parameters $(D,A)$ on $t$ is not arbitrary. Throughout the paper, we assume the following spatial and frequency decompositions hold:    
    \begin{equation} \label{DAdecomposition}
        D(x, t) =  \sum_{i=1}^I p_i(t) D_i(x),\quad A(x, t) = \sum_{j=1}^J q_j(t) A_j(x)
    \end{equation}
    where $p_i(t)$ and $q_j(t)$ are functions of the frequency $t$, and $D_i$ and $A_j$ are functions of the spatial variable $x$.
    The underlying rationale is as follows: In multispectral photoacoustics, biological tissue contains a finite collection of dominant absorbers known as chromophores. According to the Beer--Lambert mixture law, the total absorption is the sum of the contributions of these chromophores. In \eqref{DAdecomposition}, $A_j(x)$ represents the spatial concentration of the $j$th chromophore, and $q_j(t)$ represents its extinction spectrum.
    The decomposition for $A(x,t)$ thus reflects the standard spectral-unmixing model~\cite{Cox2009Chromophores,Laufer2010Chromophores}. On the other hand, the optical diffusion coefficient in the diffusion approximation is taken as $D(x,t)=(3\bigl(A(x,t)+\mu_s'(x,t)\bigr))^{-1}$, where $\mu_s'$ is the reduced scattering coefficient. Over a restricted spectral band, the frequency dependence of the optical properties is typically smooth and governed by a small number of tissue constituents or scattering mechanisms. As a result, $D(x,t)$ can be approximated by a finite number of spectral functions $p_i(t)$ with spatial coefficients $D_i(x)$.

    \textbf{Frequency-Averaged Measurement:} In diffusive frequency-resolved QPAT, the measurement at a fixed frequency $t$ is the absorbed energy $A(x,t)u(x,t)$. A broadband photoacoustic measurement averages the contribution over the frequency band, leading to the following frequency-averaged measurement:
    \begin{equation}\label{frequency-measurement}
        \mathfrak{m}(x) := \int_{\TT^d} A(x, t) u(x, t) dt = \sum^J_{j=1} A_j(x) \int_{\TT^d} u(x, t) q_j(t) dt.
    \end{equation}
    Here, $J$ is the total number of chromophores.

    \textbf{Inverse Problem:} The inverse problem in frequency-averaged diffusive QPAT is as follows: Given the frequency components $\{p_i(t)\}_{i=1}^I$ and $\{q_j(t)\}_{j=1}^J$, we aim to determine the unknown spatial components $\{D_i(x)\}_{i=1}^I$ and $\{A_j(x)\}_{j=1}^J$ by making one or more measurements of the form~\eqref{frequency-measurement}.

    \textbf{The Contribution:}
    This paper concerns the inverse problem for the frequency-averaged diffusive QPAT model. To the best of our knowledge, this model has not been systematically studied in the literature.
    We recast the imaging process as the problem of inverting a nonlinear parameter-to-measurement operator and obtain several uniqueness and stability results. 

    We first analyze the linearized frequency-averaged inverse problem. We prove that the parameter-to-measurement map is Fr\'echet differentiable and show that, for sufficiently regular coefficients and a suitable finite collection of broadband illuminations, the Fr\'echet derivative is injective and satisfies a Lipschitz stability estimate; see Theorem~\ref{main-3}. Thus, infinitesimal changes in the optical parameters are uniquely determined by the corresponding changes in the averaged internal measurements. This establishes local identifiability at the linearized level for the frequency-averaged diffusive QPAT. 

    A key step in our linearized analysis is a refined analysis of the linearized frequency-resolved diffusive QPAT problem. Previous work established stability up to a finite-dimensional kernel associated with an elliptic pseudodifferential system~\cite{Kuchment2012}, as well as generic injectivity~\cite{Kuchment2015}. We show that this kernel is trivial for sufficiently regular coefficients and suitable illuminations, thereby establishing both uniqueness and Lipschitz stability for the frequency-resolved linearized QPAT problem; see Proposition~\ref{prop}. The finite-rank spectral structure of the optical parameters then allows these conclusions to be extended to the frequency-averaged setting.

    We subsequently pass from the linearized problem to the full nonlinear inverse problem by means of a Banach-space inverse function theorem~\cite{Stefanov2009}. This yields local uniqueness and conditional H\"older stability near any sufficiently regular admissible optical parameters; see Theorem~\ref{main-2}. Consequently, optical parameters that are sufficiently close to a reference configuration cannot yield identical averaged measurements, and small measurement errors can lead to quantitatively controlled reconstruction errors, even if the light source is moderately broadband. These results provide a mathematical foundation for the local recovery of frequency-dependent optical properties from frequency-averaged QPAT data.

    The rest of the paper is organized as follows. 
    In \Cref{section-3}, we introduce the notation used throughout the paper and summarize the main results. In \Cref{section-4}, we establish the differentiability of the frequency-averaged measurement with respect to the PDE parameters and derive the linearized measurement map. The uniqueness and stability of the linearized inverse problem are discussed in \Cref{section-5}. \Cref{section-5-1} is dedicated to the classical, frequency-resolved linearized QPAT problem, while \Cref{section-5-2} is dedicated to the frequency-averaged linearized QPAT problem. Finally, \Cref{section-6} establishes a local stability estimate of H\"older type for the frequency-averaged nonlinear QPAT problem.

    \section{Main Results}\label{section-3}
    \subsection{Preliminaries}

    Throughout the paper, we assume that the decompositions~\eqref{DAdecomposition} hold, where the functions $p_i, q_j \in C(\TT^d)$ are continuous and known, and both $\{p_i\}_{i=1}^I$ and $\{q_j\}_{j=1}^J$ are linearly independent systems bounded below by a positive constant. 
    
    In our analysis, we frequently use products of Banach spaces. For example, given $K$ real-valued Banach spaces $X_1,\ldots, X_K$, we define the direct sum as
    \begin{align*}
        \bigoplus_{k=1}^K X_k  := \{u := (u_1,\ldots, u_K) : u_k\in X_k\;\text{for all}\;1\leq k\leq K\}
    \end{align*}
    with norm $\Vert u \Vert_{\oplus X_k} := \sqrt{\sum_{k=1}^K \Vert u_k\Vert_{X_k}^2}$. In the case of Sobolev spaces, we identify $W^{s,p}(\Omega; \RR^K) \cong W^{s,p}(\Omega)^K$, i.e., $K$ copies of $L^p$-type Sobolev spaces with regularity $s$, and define $W^{s,p}_0(\Omega)$ as the closure of $C_c^\infty(\Omega)$ in the $W^{s,p}(\Omega)$ norm.
    \par
    We say that a function or a family of functions is uniformly positive on $\Omega$ if there exists a fixed positive constant that lower bounds the functions on $\Omega$.
    \par
    Throughout the paper, we use $C>0$ to denote generic constants that depend only on the fixed background data of the problem (such as the Sobolev index $s$, the domain $\Omega$, and the weights $p_i$ and $q_j$) and whose values may change from line to line.
    \par
    We adopt the following notation for parametric PDEs and functions. Given $u \in C(\TT^d; H^1(\Omega))$, where $C(\TT^d; H^1(\Omega))$ denotes the space of continuous $H^1(\Omega)$-valued functions with domain $\TT^d$ (here $H^k(\Omega) = W^{k,2}(\Omega)$), and $t \in \TT^d$, we denote $u_t := u(\cdot, t) \in H^1(\Omega)$. The frequency-dependent QPAT model in \eqref{frequency-QPAT} can then be written as \begin{equation}\label{eq:Lt_2}
        \begin{cases}
            L_t u_t &= -\nabla\cdot (D_t \nabla u_t) + A_t u_t = 0,\\
            u_t\vert_{\partial\Omega} &= f_t.
        \end{cases}
    \end{equation}
    for each fixed $t\in \TT^d$, where $L_t$ is the operator $L$ evaluated at $t \in \TT^d$.

    We pose the inverse coefficient problem within the following admissible sets: given a precompact domain $\Omega'\subset\!\subset \Omega$ and constants $c_1, c_2 > 0$, we define the admissible set as
    \begin{equation}\label{admissible-set-1}
        \mathcal A := \{\varphi\in H^{s+2}(\Omega) : \varphi > c_1 \;\text{on}\;\Omega, \varphi \equiv 1\;\text{on}\; \Omega - \Omega', \;\text{and}\; \Vert \varphi\Vert_{H^{s+2}(\Omega)} < c_2\},
    \end{equation}
    for a suitable $s > n/2$.
    We also define a more strongly preconditioned subset of $\mathcal A$:
    \begin{equation}\label{admissible-set-2}
        \mathcal A_\infty := \{\varphi\in W^{s + 2, \infty}(\Omega) : \varphi > c_1 \;\text{on}\;\Omega, \varphi \equiv 1\;\text{on}\; \Omega - \Omega', \;\text{and}\; \Vert \varphi\Vert_{W^{s+2, \infty}(\Omega)} < c_2\}.
    \end{equation}
    For simplicity, we will take $s \in \NN$ when considering the set $\mathcal A_\infty$. Note that $\mathcal A_\infty \subseteq \mathcal A$. 
    
    We summarize all notations used in this paper in the following table:
    {
        \def\arraystretch {1.25}
        \begin{longtable}{ ||c||p{13cm}|| }
            \hline
            Symbols & Definitions\\
            \hline\hline
            \endhead
            ${\boldsymbol A}(x)$ & vector of attenuation, ${\boldsymbol A} = (A_1,\ldots, A_J)$\\
            \hline
            $\dot{{\boldsymbol A}}(x)$ & vector of perturbed attenuation, $\dot{{\boldsymbol A}} = (\dot A_1,\ldots, \dot A_J)$\\
            \hline
            $A(x, t), A_t(x)$ & attenuation coefficient, $A(x, t) = A_t(x) = \sum_{j=1}^J q_j(t) A_j(x)$\\
            \hline
            $\dot A(x, t), \dot A_t(x)$ & perturbed attenuation, $\dot A(x, t) = \dot A_t(x) = \sum_{j=1}^J q_j(t) \dot A_j(x)$\\
            \hline
            $\mathcal A$, $\mathcal A_\infty$ & preconditioned sets, refer to Equations~\eqref{admissible-set-1} and \eqref{admissible-set-2}\\
            \hline
            ${\boldsymbol D}(x)$ & vector of diffusivity, ${\boldsymbol D} = (D_1,\ldots, D_I)$\\
            \hline
            $\dot{{\boldsymbol D}}(x)$ & vector of perturbed diffusivity, $\dot{{\boldsymbol D}} = (\dot D_1,\ldots, \dot D_I)$\\
            \hline
            $D(x, t), D_t(x)$ & diffusion coefficient, $D(x, t) = D_t(x) = \sum_{i=1}^I p_i(t) D_i(x)$\\
            \hline
            $\dot D(x, t), \dot D_t(x)$ & perturbed diffusivity, $\dot D(x, t) = \dot D_t(x) = \sum_{i=1}^I p_i(t) \dot D_i(x)$\\
            \hline
            $G$, $G_t$ & solution operators for $L$ and $L_t$ under homogeneous boundary conditions\\
            \hline
            $L, L_t$ & PDE operator, refer to Equations~\eqref{frequency-QPAT} and \eqref{eq:Lt_2}\\
            \hline
            $\mathfrak{m}(x)$ & measurement given by \eqref{frequency-measurement}\\
            \hline
            $M$, $m$ & a single measurement operator and its linearization (dependent on $D$ and $A$), refer to Equations~\eqref{eq:measop} and \eqref{eq:measop-linear}
            \\
            \hline
            $\boldsymbol M$, $\boldsymbol m$ & system of multiple measurement operators and the system of their linearizations, given by $\boldsymbol M := (M_1, \ldots, M_K)$  and $\boldsymbol m := (m_1, \ldots, m_K)$\\
            \hline
            $\Omega$, $\Omega'$ & precompact domains in $\RR^n$, $\Omega'\subset\!\subset \Omega$\\
            \hline
            $p_i(t)$, $q_j(t)$ & uniformly positive weight functions in $C(\TT^d)$\\
            \hline
            $T$ & harmonic extension, $T\in B(H^{1/2}(\partial\Omega), H^1(\Omega))$\\
            \hline
            $u(x,t)$ & solution to Equation~\eqref{frequency-QPAT}\\
            \hline
            $v(x, t)$ & linearized solution to Equation~\eqref{linearized-frequency-QPAT}\\
            \hline
            $X$, $X'$, $X''$ & Banach spaces for the unknown variables, refer to \Cref{inverse-function-theorem}\\
            \hline
            $Y$, $Y'$, $Y''$ & Banach spaces for the observed variables, refer to \Cref{inverse-function-theorem}\\
            \hline
        \end{longtable}
    }
    \subsection{Organization of the main results}
    In this section, we outline our approach to the inverse problem. \Cref{section-4} establishes the forward regularity of the QPAT problem. In \Cref{section-4-1}, we show that given a boundary value $f \in C(\TT^d; H^{1/2}(\partial\Omega))$ and uniformly positive coefficients $({\boldsymbol D}, {\boldsymbol A}) \in L^\infty(\Omega; \RR^{I+J})$, the solution $u$ to PDE~\eqref{frequency-QPAT} is well posed in $C(\TT^d; H^1(\Omega))$. Hence, for a fixed boundary value $f$, we can define the \textit{measurement operator}
        \begin{equation} \label{eq:measop}
            M: (\boldsymbol D, \boldsymbol A)\mapsto \mathfrak{m} \in L^2(\Omega)
            \;\;\text{by}\;\;
            M({\boldsymbol D}, {\boldsymbol A}) = \mathfrak{m}(x) = \int_{\TT^d} A(x, t) u(x, t) dt.
        \end{equation}
    We furthermore show (see~\Cref{analyticity}) that $M$ is Fr\'echet differentiable in the above topology at all uniformly positive $({\boldsymbol D}, {\boldsymbol A}) \in L^\infty(\Omega; \RR^{I+J})$.
    \par
    In \Cref{section-4-2}, we study the linearized operator (or the Fr\'echet derivative) of $M$, which we denote by $m$ and which is given by 
    \begin{equation} \label{eq:measop-linear}
        m(\dot {\boldsymbol D}, \dot {\boldsymbol A}) = \int_{\TT^d} \dot A(x, t) u(x,t)dt + \int_{\TT^d} A(x, t) v(x, t)dt,
    \end{equation}
    where $v \in C(\TT^d; H^1_0(\Omega))$ satisfies the parametric PDE
    \begin{align*}
        L_t v_t = \nabla\cdot (\dot D_t(x) \nabla u_t(x)) - \dot A_t(x) u_t(x),
    \end{align*}
    where $\dot D(x, t) = \dot D_t(x):= \sum_{i=1}^I p_i(t)\dot D_i(x)$ and $\dot A(x, t) = \dot A_t(x):= \sum_{j=1}^J q_j(t)\dot A_j(x)$. We show (in~\Cref{linear-forward-stability}) that given $({\boldsymbol D}, {\boldsymbol A}) \in \mathcal A^I\times\mathcal A^J$ with $s > n/2$, the operator $m$ satisfies the mapping property:
    \begin{align*}
        m : L^2(\Omega; \RR^I) \oplus H^1(\Omega; \RR^J) \rightarrow H^1(\Omega).
    \end{align*}
    \par
    \Cref{section-5} addresses the linear inverse stability of the QPAT problem. In \Cref{section-5-1}, we revisit the frequency-independent model, which one can consider as a special case where $I = J = 1$, the weights $p$ and $q$ are constant in $\TT^d$, and the boundary value $f$ is $t$-independent, in which case the linearized operator effectively reduces to
    \begin{align*}
        m(\dot D, \dot A) = \dot A u + A v\;\;\text{where}\;\; v\in H^1_0(\Omega)\;\;\text{satisfies}\;\; Lv = \nabla\cdot(\dot D\nabla u) - \dot Au.
    \end{align*}
    Recall the result in \cite{Kuchment2012}: there are $2n$ linearized measurements determined by CGO solutions such that the system of maps
    \begin{align*}
        {\boldsymbol m} = (m_1,\ldots, m_{2n}) : L^2(\Omega') \oplus H^1_0(\Omega') \rightarrow H^1(\Omega; \RR^{2n}),\;\;\text{where}\;\;\Omega'\subset\!\subset\Omega,
    \end{align*}
    is upper semi-Fredholm (see~\Cref{semi-fredholm}). We use the properties of the CGO solutions to show that the system of linearized maps is in fact injective (see \Cref{injectivity}), hence establishing linear inverse stability for the frequency-independent QPAT problem (see \Cref{prop}).
    \par
    In \Cref{section-5-2}, we shift our attention back to the frequency-dependent QPAT problem and show that by localizing the boundary data to specific frequencies in $\TT^d$, the linearized inverse stability can be adapted to the frequency-dependent QPAT problem using $2n(I+J)$ measurements (see~\Cref{main-3}).
    \par
    In \Cref{section-6}, we employ the general nonlinear inversion framework developed in \cite{Stefanov2009} (also presented in~\Cref{inverse-function-theorem}), which deduces local nonlinear stability directly from the properties of the linearized operator. By combining the forward regularity from \Cref{section-4} with the inverse regularity from \Cref{section-5}, we establish the local nonlinear stability of the frequency-dependent QPAT problem using $2n(I+J)$ measurements $\boldsymbol M:= (M_1,\ldots, M_{2n(I+J)})$, which yields~\Cref{main-2}.
    \par
    In the next section, we explicitly state the main results of our work.
    
    \subsection{Main Results}
    
    \begin{theorem}\label{main-3}
        For the frequency-averaged QPAT problem \eqref{frequency-QPAT}, assume that $({\boldsymbol D}, {\boldsymbol A}) \in \mathcal A^I\times\mathcal A^J$ with $s > n/2$. There exist $2n(I+J)$ measurement operators $M_1,\ldots, M_{2n(I+J)}$, each of the form \eqref{frequency-measurement}, such that the corresponding system of $2n(I+J)$ linearized measurement operators, as given in \eqref{eq:measop-linear}, defines a system of bounded linear maps
        \begin{align*}
            {\boldsymbol m} = (m_1,\ldots, m_{2n(I+J)}) : L^2(\Omega'; \RR^I) \oplus H^1_0(\Omega'; \RR^J) \rightarrow H^1(\Omega; \RR^{2n(I+J)})
        \end{align*}
        that satisfies the bi-Lipschitz estimate:
        \begin{align*}
            C^{-1}\Vert {\boldsymbol m}(\dot {\boldsymbol D}, \dot {\boldsymbol A})\Vert_{H^1(\Omega; \RR^{2n(I+J)})} \leq \Vert\dot {\boldsymbol D}\Vert_{L^2(\Omega; \RR^I)} + \Vert\dot {\boldsymbol A}\Vert_{H^1(\Omega; \RR^J)} \leq C\Vert {\boldsymbol m}(\dot {\boldsymbol D}, \dot {\boldsymbol A})\Vert_{H^1(\Omega; \RR^{2n(I+J)})}.
        \end{align*}
    \end{theorem}
    The proof is given at the end of \Cref{section-5-2}. Note that the above result is derived from the linear analysis on the classical QPAT problem, which is stated in the following proposition:
    \begin{proposition} \label{prop}
        Given the QPAT problem \eqref{QPAT}, assuming $(D, A) \in \mathcal A\times\mathcal A$ with $s > n/2$, there are $2n$ measurements determined by $2n$ CGO boundary values such that the system of linearized measurements $L^2(\Omega') \oplus H^1_0(\Omega') \rightarrow H^1(\Omega; \RR^{2n})$
        satisfies Lipschitz inverse stability.
    \end{proposition}
    Note that this proposition complements the principal symbol analysis in \cite{Kuchment2012}, and the proof is provided in~\Cref{section-5-1}. Below we state our main result on the nonlinear inverse stability of the frequency QPAT problem:
    \begin{theorem} \label{main-2}
        Consider the frequency-dependent QPAT problem in \eqref{frequency-QPAT}. For any H\"older exponent $\alpha \in (1/2, 1)$, there exists a Sobolev index $s>n/2$ such that the following holds:
        \par
        For any $({\boldsymbol D}, {\boldsymbol A})\in \mathcal A_\infty^{I}\times \mathcal A_\infty^J$, there are $2n(I+J)$ boundary data $f_{1,1}, \ldots, f_{2n, I+J} \in C(\TT^d; H^{s+1/2}(\partial\Omega))$ and a relatively open neighborhood $U$ of $({\boldsymbol D}, {\boldsymbol A})$ in $\mathcal A_\infty^I \times A_\infty^J$ with respect to the $W^{s+2, \infty}(\Omega; \RR^{I+J})$ topology, on which the H\"older-type inverse stability estimate
        \begin{align*}
            \Vert({\boldsymbol D}_1, {\boldsymbol A}_1) - ({\boldsymbol D}_2, {\boldsymbol A}_2)\Vert_{L^\infty(\Omega; \RR^{I+J})} \leq C_\alpha \Vert {\boldsymbol M}({\boldsymbol D}_1, {\boldsymbol A}_1)  - {\boldsymbol M}({\boldsymbol D}_2, {\boldsymbol A}_2)\Vert_{{L}^2(\Omega; \RR^{2n(I+J)})}^\alpha,
        \end{align*}
        holds for all $({\boldsymbol D}_1, {\boldsymbol A}_1), ({\boldsymbol D}_2, {\boldsymbol A}_2)\in U$. Here, ${\boldsymbol M}:= (M_1, \ldots, M_{2n(I+J)})$ is the system of measurement operators defined in \eqref{frequency-measurement} for the $2n(I+J)$ boundary data $f_{k,l}(x, t)$.
    \end{theorem}
    The preconditioned set $\mathcal A_\infty$ depends on the Sobolev index $s$, and an asymptotic lower bound of $n/2 + O(1/n)$ is in general needed for $s$; see~\Cref{functional-setup} for further explanation.

    \section{Forward Well-Posedness and Stability}\label{section-4}
    In this section, we establish several forward regularity results for the frequency-averaged QPAT problem~\eqref{frequency-QPAT} with measurement operator~\eqref{frequency-measurement}. For simplicity, we study the case of a single measurement; the case of multiple measurements follows from the same reasoning. We consider the solution $u$ of \eqref{frequency-QPAT} in the space of $C(\TT^d; H^1(\Omega))$ given boundary data in the space $C(\TT^d; H^{1/2}(\partial\Omega))$, where a solution is defined to be the weak solution of \eqref{frequency-QPAT} pointwise in $t$ for every $t\in \TT^d$, while the measurement is considered as a Bochner integral of $Au$ on $\TT^d$. 

    This section establishes several regularity results for the forward solutions of both the frequency-dependent QPAT model \eqref{frequency-QPAT} and the frequency-independent QPAT problem \eqref{QPAT}. 

    Throughout the paper, we use $B(X, Y)$ to denote the space of bounded linear operators from given Banach spaces $X$ to $Y$, and the norm is the operator norm.

    \subsection{Differentiability of the Forward Map and Linearization}
    \label{section-4-1}

    For each $t\in\TT^d$, the existence and uniqueness of $u_t \in H^1(\Omega)$ follow from standard elliptic theory \cite{Gilbarg2001}. It remains to show the continuity of the dependence on $t$. 

    We first derive an explicit representation of the solution $u(x, t)$. Let $T:H^{1/2}(\partial\Omega) \to H^1(\Omega)$ be the harmonic extension operator; that is, $\varphi := Tf$ solves the PDE $-\Delta\varphi = 0$ with boundary value $\varphi|_{\partial\Omega} = f$. As $u_t(x) = u(x, t)$ solves \eqref{frequency-QPAT} with $u_t|_{\partial\Omega}=f_t$, we see that $u_t(x) - Tf_t(x)\in H^1_0(\Omega)$ solves the following PDE
    \begin{align*}
        L_t (u_t(x) - Tf_t(x)) &= -\nabla\cdot (D_t(x)\nabla (u_t(x) - Tf_t(x))) + A_t(x)(u_t(x) - Tf_t(x))\\
        &= - L_t\circ Tf_t(x) \in H^{-1}(\Omega).
    \end{align*}
    Now, denote by $G_t: H^{-1}(\Omega)\to H^1_0(\Omega)$ the solution operator of $L_t$ equipped with the homogeneous Dirichlet boundary condition. Applying $G_t$ yields the following explicit representation of $u_t$:
    \begin{equation} \label{eq:utrep}
    u_t = Tf_t - G_t\circ L_t\circ Tf_t.
    \end{equation}
    We use this representation to establish the forward well-posedness with $u\in C(\TT^d; H^1(\Omega))$ in the next lemma. Note that $D(x,t),A(x,t)\in C(\TT^d;L^\infty(\Omega))$, which follows directly from the decomposition~\eqref{DAdecomposition} since $p_i,q_j\in C(\TT^d)$ and $D_i,A_j\in L^\infty(\Omega)$.
    \begin{lemma}
        Let $({\boldsymbol D}, {\boldsymbol A})\in L^\infty(\Omega; \RR^{I+J})$ with $k_1 < D_i, A_j < k_2$ for all $i, j$ for some positive constants $k_1, k_2$, and let $f\in C(\TT^d; H^{1/2}(\partial\Omega))$ be a fixed boundary value. There exists a unique solution $u\in C(\TT^d; H^1(\Omega))$ to \eqref{frequency-QPAT} that satisfies the estimate
        \begin{align*}
            \Vert u \Vert_{C(\TT^d; H^1(\Omega))} = \sup\nolimits_{t\in \TT^d}\Vert u(\cdot, t)\Vert_{H^1(\Omega)} \leq C\sup\nolimits_{t\in \TT^d}\Vert f_t\Vert_{H^{1/2}(\partial\Omega)}
        \end{align*}
        for some constant $C>0$ depending only on $k_1$ and $k_2$.
    \end{lemma}
    \begin{proof}
        The norm estimate follows from \Cref{uniform-estimate} in the appendix. Here, we only need to show the continuous dependence of the solution $u_t$ on $t$. For $t, s\in \TT^d$, the representation~\eqref{eq:utrep} gives
        \begin{align*}
            u_t - u_s &= Tf_t - Tf_s + G_s\circ L_s\circ Tf_s - G_t\circ L_t \circ Tf_t\\
            &= T(f_t - f_s) + (G_s - G_t)L_sTf_s + G_t(L_s - L_t)Tf_s + G_tL_tT(f_s - f_t).
        \end{align*}
        Taking the $H^1(\Omega)$-norm shows
        \begin{align*}
            \Vert u_t - u_s\Vert_{H^1(\Omega)} &\leq \Vert T\Vert_{B(H^{1/2}(\partial \Omega), H^1(\Omega))}\Vert f_t - f_s\Vert_{H^{1/2}(\partial\Omega)} + \Vert G_s - G_t\Vert_{B(H^{-1}(\Omega), H^1(\Omega))} \Vert L_sTf_s\Vert_{H^{-1}(\Omega)}\\
            &\qquad + \Vert G_t\Vert_{B(H^{-1}(\Omega), H^1(\Omega))}\Vert L_s - L_t\Vert_{B(H^1(\Omega), H^{-1}(\Omega))}\Vert Tf_s\Vert_{H^1(\Omega)}\\
            &\qquad + \Vert G_tL_tT\Vert_{B(H^{1/2}(\partial\Omega), H^1(\Omega))} \Vert f_s - f_t\Vert_{H^{1/2}(\partial\Omega)}.
        \end{align*} 
        By \Cref{uniform-estimate}, the norms of $L_t,G_t$ are uniformly bounded in $t$ with a constant depending only on $k_1$ and $k_2$. For the difference terms, we have $\Vert f_t - f_s\Vert_{H^{1/2}(\partial\Omega)} \rightarrow 0$ as $s\rightarrow t$ using continuity in $\TT^d$.
        To estimate $L_t-L_s$, take arbitrary $\varphi_1\in H^1_0(\Omega)$, $\varphi_2\in H^1(\Omega)$ to get
        $\int_\Omega \varphi_1 (L_t-L_s) \varphi_2 dx = \int_\Omega (D_t-D_s) \nabla\varphi_1 \cdot \nabla \varphi_2 + (A_t-A_s) \varphi_1\varphi_2 \,dx$. Therefore, there exists a constant $C>0$ such that
        \begin{equation} \label{eq:Lnormdiff}
            \Vert L_t - L_s\Vert_{B(H^1(\Omega), H^{-1}(\Omega))} \leq C \left( \Vert D_t - D_s \Vert_{L^\infty(\Omega)} +  \Vert A_t - A_s \Vert_{L^\infty(\Omega)} \right),
        \end{equation}
        where the right-hand side tends to $0$ as $s\rightarrow t$, since $D(x,t), A(x,t)\in C(\TT^d;L^\infty(\Omega))$. 
        To estimate $G_t-G_s$, we apply the relations $L_s G_s = \operatorname{id}$ and $G_t L_t G_s = G_s$ on $H^{-1}(\Omega)$ (where $\operatorname{id}$ denotes the identity operator) to write
        \begin{equation} \label{eq:Gdiff}
            G_t - G_s = G_t L_s G_s - G_t L_t G_s = G_t (L_s-L_t) G_s.
        \end{equation}
        As the norms of $G_t,G_s$ are uniformly bounded in $t,s$ by \Cref{uniform-estimate}, we conclude
        \begin{align*}
            \Vert G_t - G_s\Vert_{B(H^{-1}(\Omega), H^{1}(\Omega))} \le C \Vert L_t - L_s\Vert_{B(H^1(\Omega), H^{-1}(\Omega))}.
        \end{align*}
        The right-hand side has been shown to converge to $0$ as $s\rightarrow t$. This completes the proof.
        
    \end{proof} 
    To justify linearization, we first need to show differentiability of the forward map, which we summarize in the following lemma. Recall that, given a fixed boundary value $f \in C(\TT^d; H^{1/2}(\partial\Omega))$, the \textit{measurement operator} applied to $(\boldsymbol D, \boldsymbol A)$ is given by
    \begin{align*}
        M({\boldsymbol D}, {\boldsymbol A}) =  \mathfrak{m}(x) = \int_{\TT^d} A(x, t) u(x, t) dt
    \end{align*}
    as in~\eqref{frequency-measurement} and \eqref{eq:measop}. In words, $M$ encodes the dependence of the frequency-QPAT data on the unknown parameters $({\boldsymbol D}, {\boldsymbol A})$. 
    
    \begin{lemma}\label{analyticity}
        Let $({\boldsymbol D}, {\boldsymbol A})\in L^\infty(\Omega; \RR^{I+J})$ with $k_1 < D_i, A_j < k_2$ for all $i, j$ for some positive constants $k_1, k_2$. Given a fixed boundary value $f \in C(\TT^d; H^{1/2}(\partial\Omega))$, the measurement operator
        admits the following second-order expansion for all $({\boldsymbol D}^\ast, {\boldsymbol A}^\ast)$ in an $L^\infty(\Omega; \RR^{I+J})$ neighborhood of $({\boldsymbol D}, {\boldsymbol A})$:
        \begin{align*}
            M({\boldsymbol D}^\ast, {\boldsymbol A}^\ast) = M({\boldsymbol D}, {\boldsymbol A}) + m({\boldsymbol D}^\ast - {\boldsymbol D}, {\boldsymbol A}^\ast - {\boldsymbol A}) + R({\boldsymbol D}^\ast - {\boldsymbol D}, {\boldsymbol A}^\ast - {\boldsymbol A}),
        \end{align*}
        where $m : L^\infty(\Omega; \RR^{I+J}) \rightarrow L^2(\Omega)$ is a bounded linear operator, and $R : L^\infty(\Omega; \RR^{I+J}) \rightarrow L^2(\Omega)$ satisfies the following quadratic bound in this neighborhood: 
        \begin{align*}
            \Vert R({\boldsymbol D}^\ast - {\boldsymbol D}, {\boldsymbol A}^\ast - {\boldsymbol A})\Vert_{L^2(\Omega)} \leq C(\Vert {\boldsymbol D}^\ast - {\boldsymbol D}\Vert_{L^\infty(\Omega; \RR^I)} + \Vert {\boldsymbol A}^\ast - {\boldsymbol A}\Vert_{L^\infty(\Omega; \RR^J)})^2,
        \end{align*}
        for some constant $C>0$. Furthermore, this constant $C$ depends only on the parameters $k_1, k_2$ and the boundary value $f$. 
    \end{lemma}
    \begin{proof}
        Let $({\boldsymbol D}^\ast, {\boldsymbol A}^\ast) \in L^\infty(\Omega; \RR^{I+J})$ and let $L^\ast$ and $G^\ast$ be the PDE operators corresponding to these starred coefficients. Set $\delta L_t := L_t^\ast - L_t$. Similar to~\eqref{eq:Gdiff}, we have $G_t - G^*_t = G_t \delta L_t G^*_t$, that is,
        \begin{align*}
            (\operatorname{id} + G_t \circ \delta L_t) G^*_t = G_t \qquad \text{ on } H^{-1}(\Omega).
        \end{align*}
        Similar to the derivation of~\eqref{eq:Lnormdiff}, we have
        \begin{equation}\label{eq:L_t}
            \begin{aligned}
                \Vert\delta L_t\Vert_{B(H^1(\Omega), H^{-1}(\Omega))} &\leq C\sum_{i=1}^I p_i(t) \Vert D_i - D_i^\ast \Vert_{L^\infty(\Omega)} + C\sum_{j=1}^J q_j(t) \Vert A_j - A_j^\ast \Vert_{L^\infty(\Omega)}\\
                &\leq C\Vert({\boldsymbol D}^\ast - {\boldsymbol D}, {\boldsymbol A}^\ast - {\boldsymbol A})\Vert_{L^\infty(\Omega; \RR^{I+J})}
            \end{aligned}
        \end{equation}
        where $C>0$ depends only on the constants $k_1, k_2$ and the boundary value $f$.
        This means $\Vert\delta L_t\Vert_{B(H^1(\Omega), H^{-1}(\Omega))} \ll 1$ for $({\boldsymbol D}^\ast, {\boldsymbol A}^\ast)$ sufficiently close to $({\boldsymbol D}, {\boldsymbol A})$ in $L^\infty(\Omega; \RR^{I+J})$. If so, $G_t\circ\delta L_t$ is a contraction, and the Neumann series expansion yields
        \begin{equation}\label{eq:G_t}
            G_t^\ast = \sum_{k\in \NN} (-G_t \circ \delta L_t)^k \circ G_t.
        \end{equation}
        Using this series expansion for $G^*_t$, along with the representation~\eqref{eq:utrep} for the solution $u^\ast$ corresponding to $({\boldsymbol D}^\ast, {\boldsymbol A}^\ast)$, we have
        \begin{align*}
            u_t^\ast &\phantom{:}= Tf_t- G_t^\ast (L_t^\ast Tf_t) \\
            &\phantom{:}= Tf_t - G_t(L_tTf_t) - G_t\circ \delta L_t \left(Tf_t - G_t (L_tTf_t)\right)\\
            &\phantom{:}\qquad - \sum_{k\geq 2} (-G_t \circ \delta L_t)^k \circ G_t (L_t^\ast Tf_t) + (G_t\circ \delta L_t)^2 Tf_t\\
            &\phantom{:}= u_t - G_t \circ \delta L_t u_t + \Phi_t,\\
            \text{where}\;\;\Phi_t &:= - \sum_{k\geq 2} (-G_t \circ \delta L_t)^k \circ G_t (L_t^\ast Tf_t) + (G_t\circ \delta L_t)^2 Tf_t.
        \end{align*}
        Denoting $\delta {\boldsymbol D} = (\delta D_1, \ldots, \delta D_I):= {\boldsymbol D}^\ast - {\boldsymbol D}$, $\delta D := D^\ast (x, t) - D(x, t)$,
        $\delta {\boldsymbol A} = (\delta A_1, \ldots, \delta A_J):= {\boldsymbol A}^\ast - {\boldsymbol A}$, and $\delta A := A^\ast(x, t) - A(x, t)$. Then $\Phi_t$ satisfies the estimate
        \begin{align*}
            \Vert \Phi_t \Vert_{H^1(\Omega)} &\leq \left(\sum_{k\in \NN}\Vert (G_t\circ \delta L_t)^k\circ G_t\circ L^\ast_t\Vert_{B(H^1(\Omega), H^1(\Omega))} + 1\right)\Vert G_t\circ \delta L_t\Vert_{B(H^1(\Omega), H^1(\Omega))}^2\Vert Tf_t\Vert_{H^1(\Omega)}\\
            &\leq C\Vert \delta L_t\Vert_{B(H^1(\Omega), H^{-1}(\Omega))}^2 \leq C\Vert(\delta {\boldsymbol D}, \delta {\boldsymbol A})\Vert_{L^\infty(\Omega; \RR^{I+J})}^2
        \end{align*}
        where, by \Cref{uniform-estimate} and \eqref{eq:L_t}, the constant $C>0$ can be chosen to depend only on the constants $k_1,k_2$ and the boundary value $f$. 
        Then the measurement operator $M^\ast$ can be written as
        \begin{align*}
            M^\ast(x) & = \int_{\TT^d} A^*_t(x) u^*_t(x) \,dt \\ 
            & = \int_{\TT^d} A^\ast_t \left (u_t - G_t \circ \delta L_t u_t + \Phi_t \right) dt\\
            &= \int_{\TT^d} A_t u_t dt + \int_{\TT^d} \delta A_t u_t - A_t G_t\circ \delta L_t u_t dt + \int_{\TT^d} - \delta A_t G_t\circ \delta L_t u_t + A^\ast_t \Phi_t dt\\
            &= M(x) + m(\delta {\boldsymbol D}, \delta {\boldsymbol A}) + R(\delta {\boldsymbol D}, \delta {\boldsymbol A})
        \end{align*}
        where $m$ and $R$ are defined as:
        \begin{align*}
            m(\delta {\boldsymbol D}, \delta {\boldsymbol A}) &= \int_{\TT^d} \delta A_t u_t - A_t G_t\circ \delta L_t u_t dt\\
            &= \int_{\TT^d} \sum_{j=1}^J q_j(t) \left(\delta A_j(x) u(x, t) - A_j(x) (G_t\circ \delta L_t u_t)(x, t)\right) dt\\
            R(\delta {\boldsymbol D}, \delta {\boldsymbol A}) &= \int_{\TT^d} - \delta A_t G_t\circ \delta L_t u_t + A^\ast_t \Phi_t dt \in O\left(\Vert(\delta {\boldsymbol D}, \delta {\boldsymbol A})\Vert_{L^\infty(\Omega; \RR^{I+J})}^2\right),
        \end{align*}
        Thus, $M\in L^2(\Omega)$ has the desired second-order expansion.        
    \end{proof}
    \begin{remark}
        We note that, for the frequency-independent case \eqref{QPAT}, if $I = J = 1$, $p$ and $q$ are constant unit weights, and $f_t$ is constant in $t\in \TT^d$, then the above result reduces to an analogue of \cite[Lemma 2.1]{Kuchment2012}. For a fixed boundary datum $f \in H^{1/2}(\partial\Omega)$, the measurement operator $M:L^\infty(\Omega; \RR^2) \rightarrow L^2(\Omega)$ has a bounded second-order expansion at all uniformly positive base points $(D, A) \in L^\infty(\Omega; \RR^2)$.
    \end{remark}
    \par

    We now proceed to analyze the linearized map of $(D, A) \mapsto M$. By the above \Cref{analyticity}, it is clear that the Fr\'echet derivative is well-defined. Analogous to the linearization analysis in \cite{Kuchment2012}, we perturb the PDE parameter vector $({\boldsymbol D}, {\boldsymbol A})$ by $\dot {\boldsymbol D} := (\dot D_1, \ldots, \dot D_I) \in L^\infty(\Omega; \RR^I)$ and $\dot {\boldsymbol A} := (\dot A_1, \ldots, \dot A_J) \in L^\infty(\Omega; \RR^J)$, and let $v\in C(\TT^d; H^1_0(\Omega))$ be the solution to the following PDE
    \begin{align}\label{linearized-frequency-QPAT}
        \begin{aligned}
            L_tv_t &= - \nabla\cdot (D_t(x) \nabla v_t(x)) + A_t(x) v_t(x) = \nabla\cdot (\dot D_t(x) \nabla u_t(x)) - \dot A_t(x) u_t(x)\\
            &= \sum_{i=1}^I p_i(t) \nabla\cdot (\dot D_i(x) \nabla u_t(x)) - \sum_{j=1}^J q_j(t) \dot A_j(x) u_t(x),
        \end{aligned}
    \end{align}
    where $\dot D_t(x) :=  \sum_{i=1}^I p_i(t) \dot D_i(x)$ and $\dot A_t(x) := \sum_{j=1}^J q_j(t) \dot A_j(x)$. Note that $v_t = -G_t \circ \delta L_t u_t$ directly from the first-order term of the expansion from the previous lemma. In other words, to obtain the Fr\'echet derivative of $M$, we set $\delta {\boldsymbol D} = \eps \dot {\boldsymbol D}$ and $\delta {\boldsymbol A} = \eps \dot {\boldsymbol A}$; the Fr\'echet derivative is given by
    \begin{align*}
        \lim_{\eps\rightarrow 0}\frac{1}{\eps}\left(M({\boldsymbol D} + \eps\dot {\boldsymbol D}, {\boldsymbol A} + \eps\dot {\boldsymbol A}) - M({\boldsymbol D}, {\boldsymbol A})\right) = \lim_{\eps\rightarrow 0} \frac{1}{\eps}\left(m(\eps\dot {\boldsymbol D}, \eps\dot {\boldsymbol A}) + R(\eps\dot {\boldsymbol D}, \eps \dot {\boldsymbol A})\right) = m(\dot {\boldsymbol D}, \dot {\boldsymbol A})
    \end{align*}
    Expanding out the expression of $m$, as in the last lemma, we have $m : L^\infty(\Omega; \RR^{I+J}) \rightarrow L^2(\Omega)$ given by the integral (as in \eqref{eq:measop-linear}): 
    \begin{equation*}
        m (\dot {\boldsymbol D}, \dot {\boldsymbol A}) = \sum_{j=1}^J \dot A_j(x)\int_{\TT^d} u(x,t)q_j(t)dt + \sum_{j=1}^J A_j(x)\int_{\TT^d} v(x, t) q_j(t)dt.
    \end{equation*}
    \begin{remark}
        Again, observe that taking $I = J = 1$, $p\equiv q\equiv 1$, and a fixed $t\in \TT^d$ retrieves the following linearized map for the classical QPAT problem as in \cite{Kuchment2012} (up to a constant of the Lebesgue measure of $\TT^d$):
        \begin{equation} \label{linearized-measurement}
            m (\dot D, \dot A) = \dot A(x) u(x) + A(x)v(x).
        \end{equation}
        where $v\in H^1_0(\Omega)$ satisfies the PDE
        \begin{equation}\label{linearized-QPAT}
            -\nabla\cdot (D(x)\nabla v(x)) + A(x)v(x) = \nabla\cdot (\dot D(x) \nabla u(x)) - \dot A(x) u(x).
        \end{equation}
    \end{remark}
    \subsection{Forward Regularity of the Linearized Map}\label{section-4-2}
    Here, we establish Sobolev regularity and boundedness properties of the linearized forward map under admissible assumptions on the background coefficients and boundary data. 
    \begin{lemma}\label{linear-forward-stability}
        Let $s > n/2$ and $({\boldsymbol D}, {\boldsymbol A}) \in \mathcal A^I \times \mathcal A^J$ be an admissible pair of parameters. Then for a fixed boundary datum $f \in C(\TT^d; H^{s+1/2}(\partial \Omega))$, the linearized measurement operator $m$ defined in~\eqref{eq:measop-linear} extends to a bounded linear operator
        \begin{align*}
            m: L^2(\Omega;\RR^{I+J}) \cong L^2(\Omega; \RR^I) \oplus L^2(\Omega; \RR^J) \to L^2(\Omega).
        \end{align*}
        Furthermore, $m$ restricts to a bounded linear operator
        \begin{align*}
            m: L^2(\Omega; \RR^I) \oplus H^1(\Omega; \RR^J) \to H^1(\Omega).
        \end{align*}
    \end{lemma}
    \begin{proof}
        Let $\dot {\boldsymbol D} \in L^2(\Omega; \RR^I)$ and $\dot {\boldsymbol A} \in L^2(\Omega; \RR^J)$. Using the operator $G_t$, we can write the solution $v\in C(\TT^d; H^1_0(\Omega))$ to \eqref{linearized-frequency-QPAT} as
        \begin{align*}
            v_t = G_t (\nabla \cdot (\dot D_t \nabla u_t) - \dot A_t u_t).
        \end{align*}
        By \Cref{uniform-estimate}, $G_t$ in $B(H^{-1}(\Omega), H^1(\Omega))$ is uniformly bounded in $t$. By \Cref{solution-stability-1}, $u_t$ can be bounded in $W^{1,\infty}(\Omega)$ uniformly in $t$ since $f$ is continuous in $\TT^d$.
        Therefore, there exists a constant $C>0$ such that
        \begin{equation}\label{v_t-estimate}
            \begin{aligned}
                \Vert v_t\Vert_{H^1(\Omega)} &\leq \Vert G_t\Vert_{B(H^{-1}(\Omega), H^1(\Omega))} \Vert \nabla \cdot (\dot D_t \nabla u_t) - \dot A_t u_t \Vert_{H^{-1}(\Omega)}\\
                &\leq \sup\nolimits_{t\in\TT^d}(\Vert G_t\Vert_{B(H^{-1}(\Omega), H^1(\Omega))}\Vert u_t \Vert_{W^{1,\infty}(\Omega)})(\Vert \dot {\boldsymbol D} \Vert_{L^2(\Omega; \RR^I)} + \Vert \dot {\boldsymbol A} \Vert_{L^2(\Omega; \RR^J)})\\
                &\leq C(\Vert \dot {\boldsymbol D} \Vert_{L^2(\Omega; \RR^I)} + \Vert \dot {\boldsymbol A} \Vert_{L^2(\Omega; \RR^J)})
            \end{aligned}
        \end{equation}
        and this estimate is uniform in $t\in\TT^d$. Now referring to \eqref{eq:measop-linear}, we obtain estimates on $m$
        \begin{align*}
            \Vert m(\dot {\boldsymbol D}, \dot {\boldsymbol A})\Vert_{L^2(\Omega)} &\leq \sum_{j=1}^J \Vert \dot A_j \Vert_{L^2(\Omega)}\left \Vert \int_{\TT^d} u_t q_j(t)dt \right\Vert_{L^\infty(\Omega)} + \Vert A_j \Vert_{L^\infty(\Omega)} \left\Vert \int_{\TT^d} v_t q_j(t) dt\right\Vert_{L^2(\Omega)}\\
            &\leq C\left( \sup\nolimits_{t\in\TT^d}\Vert u_t\Vert_{L^\infty(\Omega)} \sum_{j=1}^J \Vert \dot A_j \Vert_{L^2(\Omega)} + \sup\nolimits_{t\in\TT^d}\Vert v_t\Vert_{L^2(\Omega)}\Vert A_j \Vert_{H^{s+2}(\Omega)}\right)\\
            &\leq C(\Vert \dot {\boldsymbol D} \Vert_{L^2(\Omega; \RR^I)} + \Vert \dot {\boldsymbol A} \Vert_{L^2(\Omega; \RR^J)}),
        \end{align*}
        
        for some constant $C>0$. Here, the bound on $\Vert u_t\Vert_{L^\infty(\Omega)}$ is given by \Cref{solution-stability-1}, the bound on $\Vert v_t\Vert_{L^2(\Omega)}$ is given by \eqref{v_t-estimate}, and the bound $\Vert A_j\Vert_{L^\infty(\Omega)} \leq C\Vert A_j\Vert_{H^{s+2}(\Omega)}$ follows from Sobolev embedding and the upper bound $c_2$ in the definition of $\mathcal A$. This proves the first claim.
        
        To show the second claim, we simply replace the estimates on $u_t$ and $v_t$ with higher regularity when estimating $m$:
        \begin{align*}
            \Vert m(\dot {\boldsymbol D}, \dot {\boldsymbol A}) \Vert_{H^1(\Omega)} &\leq \sum_{j=1}^J \Vert \dot A_j \Vert_{H^1(\Omega)}\left \Vert \int_{\TT^d} u_t q_j(t)dt \right\Vert_{W^{1,\infty}(\Omega)} + \Vert A_j \Vert_{W^{1,\infty} (\Omega)} \left\Vert \int_{\TT^d} v_t q_j(t) dt\right\Vert_{H^1(\Omega)}\\
            &\leq C(\Vert \dot {\boldsymbol D} \Vert_{L^2(\Omega; \RR^I)} + \Vert \dot {\boldsymbol A} \Vert_{H^1(\Omega; \RR^J)}),
        \end{align*}
        where the reasoning is similar to that above. This concludes the second claim.
    \end{proof}
    We next show that the linearized measurement operator $m$ depends continuously on the choice of background coefficients $({\boldsymbol D}, {\boldsymbol A})$ in $\mathcal A^I\times \mathcal A^J$. This result is essential in establishing uniform linear inverse estimates for all base points $({\boldsymbol D}, {\boldsymbol A}) \in \mathcal A^I\times \mathcal A^J$ in later sections, for instance in the proof of \Cref{injectivity} for all base points in $\mathcal A\times \mathcal A$.
    
    \begin{lemma}\label{linear-perturbation-analysis}
        Let $s>n/2$ and $({\boldsymbol D}, {\boldsymbol A}), ({\boldsymbol D}^\ast, {\boldsymbol A}^\ast) \in \mathcal A^I \times \mathcal A^J$ be two admissible pairs of parameters. For a fixed boundary datum $f \in C(\TT^d; H^{s + 1/2}(\partial\Omega))$, we denote by $m$ and $m^\ast$ the linearized measurement operator linearized at $({\boldsymbol D}, {\boldsymbol A})$ and $({\boldsymbol D}^\ast, {\boldsymbol A}^\ast)$, respectively. Then their difference satisfies the estimate
        \begin{align*}
            \Vert m - m^\ast \Vert_{B(L^2(\Omega; \RR^I)\oplus H^1(\Omega; \RR^J), H^1(\Omega))}
            \leq C\Vert({\boldsymbol D}, {\boldsymbol A}) - ({\boldsymbol D}^\ast, {\boldsymbol A}^\ast)\Vert_{H^{s+2}(\Omega; \RR^{I + J})},
        \end{align*}
        for some constant $C>0$.
    \end{lemma}
    \begin{proof}
        Let $\dot {\boldsymbol D}\in L^2(\Omega; \RR^I)$ and $\dot {\boldsymbol A}\in H^1(\Omega; \RR^J)$. Denote the corresponding solutions by $u$ and $u^\ast$ to \eqref{frequency-QPAT}, respectively. 
        Recall that $\dot A_t(x) = \dot A(x, t) = \sum_{j=1}^J q_j(t) \dot A_j(x)$. Using~\eqref{eq:measop-linear}, we have
        \begin{align*}
            m(\dot {\boldsymbol D}, \dot {\boldsymbol A}) - m^\ast (\dot {\boldsymbol D}, \dot {\boldsymbol A}) = \int_{\TT^d} \dot A_t (u_t - u_t^\ast)dt + \int_{\TT^d} (A_t - A^\ast_t) v^\ast_t dt + \int_{\TT^d} A_t(v_t - v_t^\ast)dt.
        \end{align*}
        Now let $I_1, I_2, I_3$ denote the $H^1(\Omega)$ norms of the three terms on the right-hand side, respectively. 
        Then the bound on the first integral $I_1$ follows directly from \Cref{solution-stability-2}:
        \begin{align*}
            I_1 &\leq \int_{\TT^d}\Vert \dot A_t\Vert_{H^1(\Omega)}\Vert u_t - u_t^\ast\Vert_{W^{1,\infty}(\Omega)}dt \leq C\sup\nolimits_{t\in \TT^d}\Vert u_t - u_t^\ast\Vert_{W^{1,\infty}(\Omega)} \Vert \dot {\boldsymbol A}\Vert_{H^1(\Omega; \RR^J)}\\
            &\leq C\Vert({\boldsymbol D}, {\boldsymbol A}) - ({\boldsymbol D}^\ast, {\boldsymbol A}^\ast)\Vert_{H^{s+2}(\Omega; \RR^{I + J})}\Vert \dot {\boldsymbol A}\Vert_{H^1(\Omega; \RR^J)}.
        \end{align*}
        For $I_2$, we bound $v_t^\ast$ as in  \eqref{v_t-estimate} to get
        \begin{align*}
            I_2 &\leq \int_{\TT^d} \Vert A_t - A^\ast_t\Vert_{W^{1,\infty}(\Omega)}\Vert v_t^\ast\Vert_{H^1(\Omega)} dt \leq  C\Vert {\boldsymbol A} - {\boldsymbol A}^\ast\Vert_{W^{1,\infty}(\Omega; \RR^J)}\sup\nolimits_{t\in\TT^d}\Vert v_t^\ast\Vert_{H^1(\Omega)}\\
            &\leq C\Vert {\boldsymbol A} - {\boldsymbol A}^\ast\Vert_{H^{s+2}(\Omega; \RR^J)} (\Vert \dot {\boldsymbol D} \Vert_{L^2(\Omega; \RR^I)} + \Vert \dot {\boldsymbol A} \Vert_{H^1(\Omega; \RR^J)}),
        \end{align*}
        For $I_3$, define $g_t \in H^{-1}(\Omega)$ and $g^\ast_t \in H^{-1}(\Omega)$ as
        \begin{align*}
            g_t := \nabla \cdot(\dot D_t\nabla u_t) - \dot A_t u_t\;\;\text{and}\;\; g_t^\ast := \nabla \cdot (\dot D_t\nabla u_t^\ast) - \dot A_t u_t^\ast,
        \end{align*}
        We have $v_t - v_t^\ast = G_tg_t - G_t^\ast g_t^\ast = G_t(g_t - g_t^\ast) + (G_t - G_t^\ast)g_t^\ast$. For $g_t - g^\ast_t$, we have the following bounds from \Cref{solution-stability-2} for all $t\in \TT^d$:
        \begin{align*}
            \Vert g_t - g^\ast_t\Vert_{H^{-1}(\Omega)}  &= \Vert\nabla \cdot(\dot D_t\nabla (u_t - u_t^\ast)) - \dot A_t (u_t - u_t^\ast)\Vert_{H^{-1}(\Omega)}\\
            &\leq C\Vert u_t - u_t^\ast\Vert_{W^{1,\infty}(\Omega)} (\Vert \dot {\boldsymbol D}\Vert_{L^2(\Omega; \RR^I)} + \Vert\dot {\boldsymbol A}\Vert_{L^2(\Omega; \RR^J)})\\
            &\leq C\Vert({\boldsymbol D}, {\boldsymbol A}) - ({\boldsymbol D}^\ast, {\boldsymbol A}^\ast)\Vert_{H^{s+2}(\Omega; \RR^{I + J})}(\Vert \dot {\boldsymbol D}\Vert_{L^2(\Omega; \RR^I)} + \Vert\dot {\boldsymbol A}\Vert_{H^1(\Omega; \RR^J)}),
        \end{align*}
        As in \eqref{eq:L_t}, $g_t$ has the following uniform estimate:
        \begin{align*}
            \Vert g_t\Vert_{H^{-1}(\Omega)} &\leq \Vert\dot D_t\Vert_{L^2(\Omega)}\Vert \nabla u_t\Vert_{L^\infty(\Omega)} + \Vert\dot A_t\Vert_{L^2(\Omega)} \Vert u_t\Vert_{L^\infty(\Omega)} \leq \Vert u\Vert_{W^{1,\infty}(\Omega)}(\Vert \dot D_t\Vert_{L^2(\Omega)} + \Vert\dot A_t\Vert_{L^2(\Omega)})\\
            &\leq C(\Vert \dot {\boldsymbol D}\Vert_{L^2(\Omega; \RR^I)} + \Vert\dot {\boldsymbol A}\Vert_{H^1(\Omega; \RR^J)}).
        \end{align*}
        For the difference of the operators $G_t$ and $G^\ast_t$, using the identity \eqref{eq:Gdiff}
        \begin{align*}
            \Vert G_t - G_t^\ast \Vert_{B(H^{-1}(\Omega), H^1(\Omega))} &= \Vert G_t(L_t^\ast - L_t)G_t^\ast \Vert_{B(H^{-1}(\Omega), H^1(\Omega))} \leq C\Vert L_t^\ast - L_t\Vert_{B(H^1(\Omega), H^{-1}(\Omega))}\\
            &\leq C\Vert({\boldsymbol D} - {\boldsymbol D}^\ast, {\boldsymbol A} - {\boldsymbol A}^\ast)\Vert_{L^\infty(\Omega; \RR^{I+J})} \leq C\Vert({\boldsymbol D} - {\boldsymbol D}^\ast, {\boldsymbol A} - {\boldsymbol A}^\ast)\Vert_{H^{s+2}(\Omega; \RR^{I+J})}
        \end{align*}
        Then, by the above estimates, \eqref{v_t-estimate}, and \Cref{uniform-estimate}
        \begin{align*}
            I_3 &\leq \int_{\TT^d}\Vert A_t\Vert_{W^{1,\infty}(\Omega)} \Vert v_t - v^\ast_t\Vert_{H^1(\Omega)} dt\\
            &\leq C\sup\nolimits_{t\in \TT^d}\left(\Vert G_t - G_t^\ast\Vert_{B(H^{-1}(\Omega), H^1(\Omega))}\Vert g_t\Vert_{H^{-1}(\Omega)} + \Vert G^\ast_t\Vert_{B(H^{-1}(\Omega), H^1(\Omega))} \Vert g_t - g^\ast_t \Vert_{H^{-1}(\Omega)}\right)\\
            &\leq C\Vert({\boldsymbol D}, {\boldsymbol A}) - ({\boldsymbol D}^\ast, {\boldsymbol A}^\ast)\Vert_{H^{s+2}(\Omega; \RR^{I + J})}(\Vert \dot {\boldsymbol D}\Vert_{L^2(\Omega; \RR^I)} + \Vert\dot {\boldsymbol A}\Vert_{H^1(\Omega; \RR^J)}),
        \end{align*}
        for some constant $C > 0$. This concludes the claim.
    \end{proof}
    \subsection{Higher Regularity of the Linearized Map}
    In this section, we precondition the parameters with stronger regularity in the admissible set $\mathcal A_\infty^I \times \mathcal A_\infty^J$. We establish higher solution regularity and boundedness of the linearized operator. These estimates will be used to frame the analysis of local nonlinear stability.
    \begin{lemma}\label{linear-forward-estimate}
        Let $s > n/2$. Suppose that the base point $({\boldsymbol D}, {\boldsymbol A})$ belongs to $\mathcal A_\infty^I \times \mathcal A_\infty^J$ and that $f \in C(\TT^d; H^{s+1/2}(\partial \Omega))$. Then
        \begin{align*}
            \Vert m(\dot {\boldsymbol D}, \dot {\boldsymbol A})\Vert_{H^{s+1}(\Omega)} \leq C \left(\Vert\dot {\boldsymbol D}\Vert_{W^{s+2, \infty}(\Omega; \RR^I)} + \Vert \dot {\boldsymbol A}\Vert_{W^{s+2, \infty}(\Omega; \RR^J)}\right)
        \end{align*}
        holds for some $C>0$ uniformly for all $({\boldsymbol D}, {\boldsymbol A}) \in \mathcal A^I_\infty \times \mathcal A^J_\infty$. As a result, the linearized measurement map $m$ restricts to a bounded linear operator
        \begin{align*}
            m: W^{s+2, \infty}(\Omega; \RR^{I+J})\to H^{s+1}(\Omega).
        \end{align*}
        with a uniform norm bound for all base points $({\boldsymbol D}, {\boldsymbol A}) \in \mathcal A^I_\infty \times \mathcal A^J_\infty$.
    \end{lemma}
    \begin{proof}
        We directly take the $H^{s+1}(\Omega)$-norm on \eqref{eq:measop-linear} to obtain
        \begin{align*}
            \Vert m(\dot{\boldsymbol D}, \dot {\boldsymbol A}) \Vert_{H^{s+1}(\Omega)} &\leq C\sum_{j=1}^J \Vert \dot A_j\Vert_{W^{s+1, \infty}(\Omega)}\left\Vert\int_{\TT^d} u_tq_j(t)dt\right\Vert_{H^{s+1}(\Omega)}\\
            &\qquad + \Vert A_j\Vert_{W^{s+1, \infty}(\Omega)}\left\Vert\int_{\TT^d} v_tq_j(t)dt\right\Vert_{H^{s+1}(\Omega)}\\
            &\leq C\left(\sup\nolimits_{t\in\TT^d}\Vert u_t\Vert_{H^{s+1}(\Omega)}\Vert \dot {\boldsymbol A}\Vert_{W^{s+1,\infty}(\Omega)} + \sup\nolimits_{t\in\TT^d}\Vert v_t\Vert_{H^{s+1}(\Omega)}\right).
        \end{align*}
        The bound on $u_t$ follows from \Cref{Schauder-estimate} and continuity in $\TT^d$. For $v_t$, by \Cref{Schauder-estimate} again:
        \begin{align*}
            \Vert v_t\Vert_{H^{s+1}(\Omega)} \leq C\Vert \nabla \cdot (\dot D_t \nabla u_t) - \dot A_t u_t \Vert_{H^{s-1}(\Omega)} \leq C\Vert u_t \Vert_{H^{s+1}(\Omega)}\Vert (\dot {\boldsymbol D}, \dot {\boldsymbol A}) \Vert_{W^{s,\infty}(\Omega; \RR^{I+J})}
        \end{align*}
        which proves the claim. The uniformity of the constant for all $({\boldsymbol D}, {\boldsymbol A}) \in \mathcal A^I_\infty \times \mathcal A^J_\infty$ is also given by \Cref{Schauder-estimate}.
    \end{proof}
    \begin{remark}
        In fact, the proof of Lemma~\ref{linear-forward-estimate} shows the stronger result that $m$ extends to a bounded linear operator 
        $m: W^{s+1, \infty}(\Omega; \RR^{I+J})\to H^{s+1}(\Omega)$.
        Here, we use the restricted domain $W^{s+2,\infty}(\Omega; \RR^{I+J})$ in order to match the choice of spaces later in \Cref{section-6}. Also note that $s+2$ regularity is needed on the parameters for CGO constructions in \Cref{cgo-1}.        
    \end{remark}
    \section{Linearized Inverse Well-Posedness and Stability}\label{section-5}
    In this section, we study the linearized inverse problem for the frequency QPAT model. We show that with multiple measurements available, the system of linearized operators is invertible and satisfies Lipschitz inverse stability. We first establish this claim for the classical QPAT problem by strengthening existing linear analysis results from \cite{Kuchment2012, Kuchment2015}; we then use the frequency-independent result to prove the frequency-dependent case. 
    
    \subsection{Linear Inverse Stability for Frequency-Independent QPAT}\label{section-5-1}

    In this section, we establish the inverse stability result for the frequency-independent linearized QPAT problem induced by \eqref{QPAT}. In other words, we take $I = J = 1$ with $p = q = 1$, and we use the notation $D, A, L$ and $G$ in place of $D_t, A_t, L_t$ and $G_t$. We assume the base point coefficients satisfy $(D, A)\in \mathcal A \times \mathcal A$ for $s > n/2$, and we consider the perturbation $(\dot D, \dot A) \in L^2(\Omega') \oplus H_0^1(\Omega')$, where $\Omega'\subset\!\subset \Omega$ is a precompact subset.
    \par
    Recall the $2n$ real CGO solutions considered in \cite{Bal2010}. 
    By \Cref{cgo-3}, there exist $2n$ real solutions $w_k\in H^{s+1}(\Omega)$ to the Schr\"odinger equation such that the matrix field given by $(\nabla w_{2k-1} w_{2k} - w_{2k-1} \nabla w_{2k})_{k=1}^n \in L^\infty(\Omega; \RR^{n\times n})$ is invertible with bounded inverse on $\Omega$.
    Then $u_k := \frac{1}{\sqrt{D}} w_k \in H^{s+1}(\Omega)$ gives $2n$ real solutions to \eqref{QPAT}, $k = 1,\dots,2n$. We set 
    \begin{align*}
        \beta_k := \nabla u_{2k-1} u_{2k} - u_{2k-1} \nabla u_{2k} \;\; \text{on} \;\; \Omega, \;\; \qquad k = 1,\dots, n,
    \end{align*}
    which are $n$ vector fields defined by the solutions $u_k$.
    \begin{lemma} \label{lem:beta}
        Let $s>n/2$ and $(D, A)\in \mathcal A \times \mathcal A$. Then 
        \begin{align*}
        \beta := (\beta_1,\ldots, \beta_n) \quad \in L^\infty(\Omega; \RR^{n\times n})
        \end{align*}
        defines an invertible matrix field on $\Omega$ and satisfies
        \begin{equation} \label{CGO-vector-field}
            \nabla D\cdot \beta_k = - D(\Delta u_{2k-1} u_{2k} - u_{2k-1}\Delta u_{2k})\;\;\text{for each}\;\; k.
        \end{equation}
    \end{lemma}
    \begin{proof}
        We directly insert $u_{2k-1} = \tfrac{1}{\sqrt{D}} w_{2k-1}$ and $u_{2k} = \tfrac{1}{\sqrt{D}} w_{2k}$ into the definition of $\beta_k$:
        \begin{align*}
            \beta_k &= \nabla u_{2k-1} u_{2k} - u_{2k-1} \nabla u_{2k} \\
            & = \nabla \left(\frac{1}{\sqrt{D}}\right)w_{2k-1}\frac{1}{\sqrt{D}}w_{2k} + \frac{1}{D}\nabla w_{2k-1}w_{2k} - \nabla \left(\frac{1}{\sqrt{D}}\right)w_{2k}\frac{1}{\sqrt{D}}w_{2k-1} - \frac{1}{D}\nabla w_{2k}w_{2k-1}\\
            &= \frac{1}{D}(\nabla w_{2k-1}w_{2k} - w_{2k-1} \nabla w_{2k}).
        \end{align*}
        As $D$ is lower and upper bounded by positive constants on $\Omega$, the matrix field $\beta$ is invertible by \Cref{cgo-3}.
        The identity follows from the elliptic PDE $L u_k = - \nabla D\cdot\nabla u_k - D\Delta u_k + Au_k = 0$; expanding $u_{2k-1}Lu_{2k} - u_{2k}Lu_{2k-1}$ gives the identity above.
    \end{proof}
    We define measurement operators $M_k = A(x)u_k(x)$ whose linearizations are $m_k(\dot D, \dot A) = \dot A(x) u_k(x) + A(x) v_k(x)$, where $v_k$ solves the linearized equation \eqref{linearized-QPAT} with $u = u_k$ for each of the $2n$ CGO solutions. Then the system of linearized measurement operators can be collected as
    \begin{equation}\label{eq:linearized-system}
        \begin{aligned}
            &{\boldsymbol m} := (m_k)_{k=1}^{2n} = (m_1,\ldots, m_{2n}) : L^2(\Omega)\oplus H^1_0(\Omega) \rightarrow H^1(\Omega; \RR^{2n}),\\
            &\text{where each } m_k \text{ corresponds to the CGO boundary data } f_k = u_k\vert_{\partial\Omega}.
        \end{aligned}
    \end{equation}
    We now show the inverse stability of this system of linearized measurement operators. To apply standard pseudodifferential operator theory, we assume for the moment that the base point $(D, A)$ is a $C^\infty(\overline\Omega)$ admissible pair. The result extends to all $(D, A)\in \mathcal A \times\mathcal A$ by \Cref{linear-perturbation-analysis}; see \Cref{perturbation-remark}. Under this assumption, \cite[Theorem 4.1]{Kuchment2012} gives the following inverse stability result through pseudodifferential operator analysis. Recall that $\Omega'\subset\!\subset \Omega$ is a precompact subset.
    \begin{lemma}[{\cite[Theorem 4.1]{Kuchment2012}}]\label{semi-fredholm}
        Given the above system of $2n$ measurements in \eqref{eq:linearized-system}, a smooth admissible pair $(D, A)\in (\mathcal A \times \mathcal A) \cap C^\infty(\overline\Omega)^2$, and $s > n/2$, the linearized operator
        \begin{align*}
            {\boldsymbol m} : L^2(\Omega') \oplus H^1_0(\Omega') \to H^1(\Omega; \RR^{2n}), \qquad (\dot D, \dot A) \mapsto (m_1(\dot D, \dot A),\ldots, m_{2n}(\dot D, \dot A))
        \end{align*}
        is upper semi-Fredholm, i.e., the image of ${\boldsymbol m}$ is closed in $H^1(\Omega; \RR^{2n})$ and the kernel of ${\boldsymbol m}$ is finite-dimensional in $L^2(\Omega')\oplus H^1_0(\Omega')$, and the following Lipschitz inverse stability estimate holds:
        \begin{align*}
            \Vert (\dot D, \dot A)\Vert_{L^2(\Omega')\oplus H^1(\Omega') / \ker{\boldsymbol m}} \leq C\Vert {\boldsymbol m}(\dot D, \dot A)\Vert_{H^1(\Omega; \RR^{2n})}
        \end{align*}
        for some constant $C>0$.
    \end{lemma}
    \begin{remark} \label{perturbation-remark}
        Note that here we consider $L^2(\Omega') \oplus H^1_0(\Omega')$ as a subspace of $L^2(\Omega) \oplus H^1(\Omega)$ and restrict ${\boldsymbol m}$ to the smaller domain. The semi-Fredholm property directly implies the inverse stability estimate modulo the kernel of the linearized map ${\boldsymbol m}$.
        \par
        We also observe that the above inverse stability constant $C$ does not essentially depend on the smoothness of $D$ and $A$. In other words, the constant $C$ can be chosen uniformly for all smooth $(D, A)\in \mathcal A\times\mathcal A$.
        Recall the analysis of the principal symbol of a coupled measurement $(m_{2k-1}, m_{2k})$ \cite[Equation (49)]{Kuchment2012}, Assuming smooth coefficients for simplicity, this symbol is determined by the matrix
        \begin{align*}
            \sigma_{\operatorname{pr}} \begin{pmatrix}
                m_{2k-1} \\ m_{2k}
            \end{pmatrix}(x,\xi) &= \begin{bmatrix}
                \frac{i\xi\cdot \nabla u_{2k-1}(x)}{\vert\xi\vert^2} & u_{2k-1}(x)\\
                \frac{i \xi \cdot \nabla u_{2k}(x)}{\vert \xi \vert^2} & u_{2k}(x)
            \end{bmatrix},\\
            \text{while}\quad \det \sigma_{\operatorname{pr}}
            \begin{pmatrix}
                m_{2k-1} \\ m_{2k}
            \end{pmatrix} (x, \xi)
            &= \frac{i}{\vert\xi\vert^2}\left(\xi\cdot(\nabla u_{2k-1}(x) {u}_{2k}(x) - u_{2k-1}(x) \nabla u_{2k}(x))\right),
        \end{align*}
        It is therefore evident that the constant $C$ in the inverse bound of ${\boldsymbol m}$ is determined by the $W^{1,\infty}(\Omega)$ norms of the functions $u_k$ and the matrix field $\beta$ (for details, see \cite[Section 4]{Kuchment2012}). By \Cref{uniform-estimate-2} and \Cref{cgo-3}, $C$ can be chosen uniformly over all smooth $(D, A) \in \mathcal A\times\mathcal A$.
    \end{remark}
    To establish inverse stability, it remains to show that the kernel of the above map on $L^2(\Omega') \oplus H^1_0(\Omega')$ is trivial, which establishes linearized inverse uniqueness: 
    \begin{proposition}\label{injectivity}
        Given the above system of $2n$ measurements determined by well-chosen CGO boundary data as in \eqref{eq:linearized-system}, with $(D, A)\in \mathcal A \times \mathcal A$ and $s > n/2$, the linearized operator defined above
        is injective on the space of $L^2(\Omega') \oplus H^1_0(\Omega')$.
    \end{proposition}
    \begin{proof}
        Our method mirrors \cite[Section 3.2]{Bal2010}. We rewrite the perturbed measurement operator as follows. Let $G = L^{-1}$ denote the solution operator to $L$ with $G : H^{-1}(\Omega) \rightarrow H^1_0(\Omega)$. Given a solution $u$ to \eqref{QPAT}, since $\dot A$ is compactly supported, $\dot A u/A \in H^1_0(\Omega)$ lies in the image of $G$. Therefore, any single linearized measurement $m$ defined in \eqref{linearized-measurement} can be expressed as
        \begin{align*}
            m &= \dot A(x) u(x) + A(x) v(x) = AG\left(L\left(\frac{\dot A}{A} u\right) + \nabla\cdot(\dot D\nabla u) - \dot A u\right)\\
            &= A G\left(-\nabla \cdot (D\nabla\left(\dot A u/ A\right))+ \nabla \cdot (\dot D\nabla u)\right),
        \end{align*}
        where we used the representation $v=G (\nabla\cdot(\dot D\nabla u) - \dot A u)$ following from \eqref{linearized-QPAT}.
        Since $G$ is injective and $A > 0$ on $\Omega$, for a pair of perturbations $(\dot D, \dot A) \in L^2(\Omega') \oplus H^1_0(\Omega')$, $m_k = 0$ for all $k$ if and only if $\dot D$ and $\dot A$ satisfy the following PDE system for all $1\leq k\leq 2n$
        \begin{equation} \label{eq:ukeqn}
            -\nabla \cdot \left(D\nabla\left(\frac{\dot A}{A} u_k \right)\right)+ \nabla \cdot (\dot D\nabla  u_k) = 0,
        \end{equation}
        Expanding this expression into terms involving $u_k$, $\nabla u_k$, and $\Delta u_k$, we obtain
        \begin{align*}
            \nabla u_k \cdot \left(\nabla \dot D - \nabla\left(D \frac{\dot A}{A}\right) - D \nabla\left(\frac{\dot A}{A}\right)\right) = \Delta u_k \left( D \frac{\dot A}{A} - \dot D\right) + u_k\nabla\cdot\left(D\nabla\left(\frac{\dot A}{A}\right)\right).
        \end{align*}
        Denote the expression on the left-hand side (and hence also on the right-hand side) by $B_k$. Computing the combination $B_{2k-1} u_{2k} - u_{2k-1}B_{2k}$ using the left-hand and right-hand expressions for $B_{2k-1}$ and $B_{2k}$, respectively, we obtain the following two equalities:
        \begin{align*}
            B_{2k-1} u_{2k} - u_{2k-1}B_{2k} &= (u_{2k}\nabla u_{2k-1} - u_{2k-1}\nabla u_{2k})\cdot \left(\nabla \dot D - \nabla\left(D \frac{\dot A}{A}\right) - D \nabla\left(\frac{\dot A}{A}\right)\right)\\
            &= \beta_k \cdot \left(\nabla \dot D - \nabla\left(D\frac{\dot A}{A}\right) - D\nabla\left(\frac{\dot A}{A}\right)\right),\\
            B_{2k-1} u_{2k} - u_{2k-1}B_{2k} &= (\Delta u_{2k-1}u_{2k} - \Delta u_{2k}u_{2k-1}) \left( D \frac{\dot A}{A} - \dot D\right)\\
            &= -\beta_k \cdot \frac{\nabla D}{D}\left( D \frac{\dot A}{A} - \dot D\right),
        \end{align*}
        where we used \eqref{CGO-vector-field} in deriving the second equality. Hence, we have the following matrix equation
        \begin{align*}
            \beta \cdot \left(\nabla \dot D - \nabla\left(D\frac{\dot A}{A}\right) - D\nabla\left(\frac{\dot A}{A}\right)\right) = -\beta \cdot \frac{\nabla D}{D} \left(D\frac{\dot A}{A} - \dot D\right).
        \end{align*}
        By Lemma~\ref{lem:beta}, $\beta$ defines an invertible matrix at almost every point in $\Omega$. Applying $\beta^{-1}$, we see that $\dot D$ and $\dot A$ solve the following first-order equation
        \begin{align*}
            \nabla \dot D - \nabla\left(D\frac{\dot A}{A}\right) - D \nabla\left(\frac{\dot A}{A}\right) = \frac{\nabla D}{D}\left(\dot D - D\frac{\dot A}{A}\right)
        \end{align*}
        in the sense of distributions,
        or after simplifying:
        \begin{align*}
            \nabla\left( \frac{\dot A}{A}\right) = \frac{1}{2}\nabla\left(\frac{\dot D}{D}\right),
        \end{align*}
        since $\dot A = \dot D = 0$ on $\Omega - \Omega'$, this immediately implies $\dot A/A = \dot D/(2D)$ on $\Omega$. Inserting $\dot{D} = 2D \dot{A}/A$ in~\eqref{eq:ukeqn}, we see that
        $\dot A / A$ solves the partial differential equation
        \begin{align*}
            -\nabla\cdot \left(D \nabla \left(u_k \frac{\dot A}{A}\right)\right) + 2\nabla\cdot\left(D \frac{\dot A}{A} \nabla u_k\right) = 0
        \end{align*}
        for all $1\le k \le 2n$.
        After some rearranging, we obtain the equation
        \begin{align*}
            -u_k\nabla\cdot \left(D\nabla \frac{\dot A}{A}\right) + \frac{\dot A}{A}\nabla\cdot (D\nabla u_k) = -u_k\nabla\cdot \left(D\nabla \frac{\dot A}{A}\right) + \frac{\dot A}{A}(Au_k) = u_kL\left(\frac{\dot A}{A}\right) = 0.
        \end{align*}
        Hence we see that $\dot A/A$ solves the PDE defined by $L$ with zero trace value. By \Cref{cgo-1}, $u_k = \frac{1}{\sqrt{D}} w_k$ is nonzero a.e. on $\Omega$ for sufficiently large $\xi$. It follows immediately that $\dot A = 0$. 
    \end{proof}
    Hence, combining the previous two results, \Cref{semi-fredholm} and \Cref{injectivity}, we obtain the linearized inverse stability result in \Cref{prop}: 
    \begin{proof}[Proof of~\Cref{prop}]
        We first observe that the claim holds for smooth base point parameters $D$ and $A$. By \Cref{semi-fredholm} and \Cref{injectivity}, for $s > n/2$ and all smooth $(D, A)\in \mathcal A \times\mathcal A \cap C^\infty(\overline\Omega)^2$, given the system of $2n$ measurements in \eqref{eq:linearized-system}, there exists a constant $C>0$ such that the system of maps ${\boldsymbol m} := (m_k)_{k=1}^{2n} = (m_1,\ldots, m_{2n}) : L^2(\Omega')\oplus H^1_0(\Omega') \rightarrow H^1(\Omega; \RR^{2n})$ satisfies the bi-Lipschitz estimate:
        \begin{align*}
            C^{-1}\Vert {\boldsymbol m}(\dot D, \dot A)\Vert_{H^1(\Omega; \RR^{2n})} \leq \Vert\dot D\Vert_{L^2(\Omega)} + \Vert\dot A\Vert_{H^1(\Omega)} \leq C\Vert {\boldsymbol m}(\dot D, \dot A)\Vert_{H^1(\Omega; \RR^{2n})},
        \end{align*}
        where the constant $C$ can be chosen uniformly, as shown in \Cref{perturbation-remark}.
        
        We now argue that the smoothness assumption on $(D, A)$ can be lifted, i.e., the linear inverse bound holds everywhere in the admissible set (cf.~\cite[Proposition 2(b)]{Stefanov2009}). This essentially follows from \Cref{linear-perturbation-analysis}.
        \par
        Let $(D^\ast, A^\ast) \in \mathcal A\times\mathcal A$ be a pair of base point coefficients and denote by ${\boldsymbol m}^\ast$ the system of $2n$ linearized measurements corresponding to $(D^\ast, A^\ast)$. Given any $\eps>0$, one can always pick $(D, A)\in \mathcal A\times\mathcal A$, with $D, A\in C^\infty(\overline\Omega)$ such that $\Vert(D-D^\ast, A-A^\ast)\Vert_{H^{s+2}(\Omega; \RR^2)} \leq \eps$. Then, by \Cref{linear-perturbation-analysis} we have
        \begin{align*}
            \Vert {\boldsymbol m} - {\boldsymbol m}^\ast \Vert_{B(L^2(\Omega)\oplus H^1(\Omega), H^1(\Omega; \RR^{2n}))} \leq C\Vert(D-D^\ast, A-A^\ast)\Vert_{H^{s+2}(\Omega; \RR^2)} \leq C\eps
        \end{align*}
        Hence, we have the inverse estimate
        \begin{align*}
            \Vert {\boldsymbol m}^\ast(\dot D, \dot A)\Vert_{H^1(\Omega; \RR^{2n})} &\geq \Vert {\boldsymbol m}(\dot D, \dot A)\Vert_{H^1(\Omega; \RR^{2n})} - C\eps (\Vert\dot D\Vert_{L^2(\Omega)} + \Vert\dot A\Vert_{H^1(\Omega)})\\
            &\geq (1/C - C\eps)(\Vert\dot D\Vert_{L^2(\Omega)} + \Vert\dot A\Vert_{H^1(\Omega)}),
        \end{align*}
        Choosing $\eps$ sufficiently small establishes linear inverse stability at all $\mathcal A\times\mathcal A$ base points. This concludes the proof of \Cref{prop}.
    \end{proof}
    \subsection{Linear Inverse Stability for the Frequency QPAT}\label{section-5-2}
    In this section, we will establish the linear inverse stability for the frequency-averaged QPAT problem using the techniques for the classical model developed in the previous section. Here we assume the base point coefficients satisfy $({\boldsymbol D}, {\boldsymbol A}) \in \mathcal A^I \times \mathcal A^J$.

    \subsubsection{Diagonalization of Parameters}
    Since $\dot D_t(x)$ and $\dot A_t(x)$ are finite linear combinations of spatial parameters $\dot D_i$'s and $\dot A_j$'s, knowing $\dot D_t$ and $\dot A_t$ at sufficiently many distinct $t \in \TT^d$ provides enough information to recover $\dot {\boldsymbol D} = (\dot D_i)_{i=1}^I$ and $\dot {\boldsymbol A} = (\dot A_j)_{j=1}^J$ uniquely. In other words, the linearized inverse problem is well-posed once enough pointwise-in-frequency data are available.
    \par
    Indeed, because the systems of frequency functions $p_i$ and $q_j$ are each linearly independent, there exists a collection of frequencies $\tau = (\tau_1,\ldots, \tau_{I+J})$ with $\tau_1, \ldots, \tau_{I+J} \in \TT^d$ such that the matrices 
    \begin{align*}
        P_{li} := p_i(\tau_l) \; \in \RR^{(I+J)\times I}, \qquad
        Q_{lj} := q_j(\tau_l) \; \in \RR^{(I+J)\times J}
    \end{align*}
    are injective with left inverses, i.e., there exist $P^\dagger \in \RR^{(I \times (I + J))}$ and $Q^\dag \in \RR^{J \times (I+J)}$ such that $P^\dagger P = \operatorname{id}_I$ and $Q^\dagger Q = \operatorname{id}_J$ where $\operatorname{id}_n \in \RR^{n\times n}$ denotes the identity matrix. Then $\dot D_{\tau_l}$ and $\dot A_{\tau_l}$ are given by
    \begin{align*}
        \dot D_{\tau_l}(x) = \dot D(x, \tau_l) = \sum_{i=1}^I P_{li}\dot D_i(x) \;\; \text{and}\;\; \dot A_{\tau_l}(x) = \dot A(x, \tau_l) = \sum_{j=1}^J Q_{lj}\dot A_j(x), \;\; l=1,\dots I+J.
    \end{align*}
    In matrix form, these relations mean 
    \begin{equation} \label{eq:matrixform}
    \dot {\boldsymbol D} = P^\dagger (\dot{D}_{ \tau_l })^{I + J}_{l = 1}, \qquad
    \dot {\boldsymbol A} = Q^\dagger (\dot{A}_{ \tau_l })^{I + J}_{l = 1}.
    \end{equation}
    Now, fix the frequencies $t=\tau_l\in \TT^d$, $l = 1,\dots, I + J$, and assume we are given $2n(I+J)$ boundary data $f_{k,l} \in C(\TT^d; H^{1/2}(\partial\Omega))$ indexed with $k = 1, \dots, 2n$ and $l = 1,\ldots, I+J$. For each $k,l$, let $u_{k,l}(x, t)$ be the solution to \eqref{frequency-QPAT} with the boundary value $f_{k,l}(x, t)$, and let $v_{k,l}(x, t)$ solve \eqref{linearized-frequency-QPAT} with $u = u_{k,l}$.
    Then the linearized measurement operator at frequency $t \in \TT^d$ corresponding to the boundary value $f_{k,l}$ is
    \begin{equation}\label{eq:m_t}
        \begin{aligned}
            m_{k,l}(\dot {\boldsymbol D}, \dot {\boldsymbol A}, t) &:= \dot A_t(x)u_{k,l}(x, t) + A_t(x) v_{k,l}(x, t)\\
            &\phantom{:}= \sum_{j=1}^J q_j(t)\left(\dot A_j(x) u_{k,l}(x, t) + A_j(x)v_{k,l}(x, t)\right).
        \end{aligned}
    \end{equation}
    For each fixed $\tau_l\in \TT^d$, suppose, as in \Cref{prop}, that the $2n$ boundary data $f_{k,l}(x, t)$, $1\leq k\leq 2n$, are chosen so that the system $(m_{k,l}(\dot {\boldsymbol D}, \dot {\boldsymbol A}, \tau_l))_{k=1}^{2n}$ satisfies 
    \begin{equation}\label{pointwise-inverse}
        \Vert(\dot D_{\tau_l}, \dot A_{\tau_l})\Vert_{L^2(\Omega)\oplus H^1(\Omega)} \leq C\Vert (m_{k,l}(\dot {\boldsymbol D}, \dot {\boldsymbol A}, \tau_l))_{k=1}^{2n} \Vert_{H^1(\Omega; \RR^{2n})}.
    \end{equation}
    Repeating this argument for all $1\leq l \leq I + J$ and letting $(m_{k,l}(\;\cdot\;, \;\cdot\;, \tau_l))_{1\leq k\leq 2n, 1\leq l \leq I + J}$ be the system of $2n(I + J)$ pointwise-in-frequency linearized measurements, we have by~\eqref{eq:matrixform} that
    \begin{equation}\label{eq:pointwise-inverse}
        \begin{aligned}
            \Vert(\dot {\boldsymbol D}, \dot {\boldsymbol A})\Vert_{L^2(\Omega; \RR^I) \oplus H^1(\Omega; \RR^J)} &= \Vert P^\dagger (\dot D_{\tau_l})_{l=1}^{I+J}, Q^\dagger (\dot A_{\tau_l})_{l=1}^{I+J}\Vert_{L^2(\Omega; \RR^I) \oplus H^1(\Omega; \RR^J)}\\ &\leq C\sum_{l=1}^{I+J} \Vert(\dot D_{\tau_l}, \dot A_{\tau_l})\Vert_{L^2(\Omega) \oplus H^1(\Omega)}\\
            &\leq C\sum_{l=1}^{I+J}\Vert (m_{k,l}(\dot {\boldsymbol D}, \dot {\boldsymbol A}, \tau_l) )_{k=1}^{2n}\Vert_{H^1(\Omega; \RR^{2n})}
        \end{aligned}
    \end{equation}
    for some constant $C>0$.

    The argument above shows that if \eqref{pointwise-inverse} is satisfied for all $\tau_l$ pointwise, then the pointwise-in-frequency linearized measurements $m_{k,l}(\;\cdot\;, \;\cdot\;, \tau_l)$ defined in \eqref{eq:m_t} recover $(\dot{\boldsymbol D}, \dot{\boldsymbol A})$ uniquely with Lipschitz stability.
    To show inverse stability for the linearized measurements defined in \eqref{eq:measop-linear}, we must analyze the integral average of \eqref{eq:m_t} over $\TT^d$.
    We address these technicalities in the next section. To justify \eqref{pointwise-inverse}, we again choose the boundary values according to the CGO solutions. We estimate the integral of \eqref{eq:m_t} over $\TT^d$ by localizing the boundary data $f_{k,l}$ on a neighborhood of $\tau_l$ for all $k$ and $l$, which we discuss in detail in the next section.
    \begin{remark}
        Note that if the two systems $p_i$ and $q_j$ are not independent, fewer measurements are generally needed to obtain the left inverses $P^\dagger$ and $Q^\dagger$. For instance, if $I=J$ and $p_i = q_i$, then $2nI$ measurements are sufficient.
    \end{remark}
    \subsubsection{Linear Inverse Stability with Averaging in Multiple Frequencies}
    For a fixed frequency $\tau_l\in \TT^d$, we choose the boundary data $f_{k,l}$ of the form
    \begin{equation}\label{frequency-boundary-data}
         f_{k,l}(x, t) := g_{k,l}(x)\eta_\eps(t - \tau_l),
    \end{equation}
    for $1\leq k\leq 2n$, where $g_{k,l}\in H^{s+1/2}(\partial\Omega)$ are the $2n$ CGO boundary data defined for the frequency-independent QPAT problem with the diffusion and attenuation coefficients taken as $(D_{\tau_l}, A_{\tau_l})$, as in \eqref{eq:linearized-system} and \Cref{semi-fredholm}. The estimate \eqref{pointwise-inverse} is automatically satisfied by \Cref{prop}, and $\{\eta_\eps\} \subseteq C(\TT^d)$ is given an approximate identity on $\TT^d$. Specifically, we can let $\eta_\eps$ be defined with, for instance, the periodized heat kernel on $\TT^d$ given by 
    \begin{align*}
        \eta_\eps(t) := C_\eps \sum_{k\in \ZZ^d}\exp(-\vert t-k\vert^2/\eps),
    \end{align*}
    where $C_\eps > 0$ is a normalizing constant such that $\int_{\TT^d} \eta_\eps (t) dt = 1$. Note that $\eta_\eps$ is strictly positive on $\TT^d$. Then for any fixed $t\in \TT^d$, $u_{k,l}(x,t)$ satisfies
    \begin{align*}
        \left\{
        \begin{aligned}
            -\nabla \cdot (D_t(x)\nabla u_{k,l}(x, t)) + A_t(x) u_{k,l}(x, t) &= 0\quad&&\text{in}\quad\Omega,\\
            u_{k,l}(x, t) \vert_{\partial\Omega} &= g_{k,l}(x)\eta_{\eps}(t-\tau_l) \quad&&\text{on}\quad\partial\Omega,
        \end{aligned}
        \right.
    \end{align*}
    and we let the reweighted linearized measurement be defined as
    \begin{equation}\label{eq:tilde m_t}
        \tilde m_{k,l}(\dot {\boldsymbol D}, \dot {\boldsymbol A}, t) := \frac{1}{\eta_\eps(t-\tau_l)} m_{k,l}(\dot {\boldsymbol D}, \dot {\boldsymbol A}, t) = \frac{1}{\eta_\eps(t-\tau_l)}\left(\dot A_t(x) u_{k,l}(x, t) + A_t(x) v_{k,l}(x, t)\right),
    \end{equation}
    where for each fixed $t\in \TT^d$, $v_{k,l}\in H^1_0(\Omega)$ satisfies 
    \begin{align*}
        L_t v_{k,l}(\cdot, t) = -\nabla\cdot (D_t\nabla v_{k,l}) + A_t v_{k,l} = \nabla\cdot (\dot D_t \nabla u_{k,l}) - \dot A_t u_{k,l}.
    \end{align*}
    Then the linearized operator $m_{k, l}$ can be decomposed as
    \begin{equation}\label{eq:m_kl}
        \begin{aligned}
            m_{k, l}(\dot {\boldsymbol D}, \dot {\boldsymbol A}) &= \sum_{j=1}^J \dot A_j(x) \int_{\TT^d} u_{k,l}(x, t)q_j(t)dt + \sum_{j=1}^J A_j(x) \int_{\TT^d} v_{k,l}(x, t) q_j(t)dt\\
            &= \int_{\TT^d} \left(\dot A_t(x) u_{k,l}(x, t) + A_t(x) v_{k,l}(x, t)\right)dt \\
            & = \int_{\TT^d} \eta_\eps(t - \tau_l) \tilde m_{k,l}(\dot {\boldsymbol D}, \dot {\boldsymbol A}, t)dt.
        \end{aligned}
    \end{equation}
    If we can show that the above weighted average of $\tilde m_{k,l}$ provides a lower bound for the single frequency measurement $m_{k,l}(\;\cdot\;,\;\cdot\;, \tau_l)$ defined in \eqref{eq:m_t} at $\tau_l$, then we can claim the inverse stability estimate holds with $2n(I+J)$ measurements. To do so, we establish the following comparison result:
    \begin{lemma}\label{approximate-identity}
        Let $({\boldsymbol D}, {\boldsymbol A}) \in \mathcal A^I \times \mathcal A^J$ with $s > n/2$. For fixed $k$ and $l$, let $m_{k,l}$ and $\tilde m_{k,l}$ be defined by \eqref{eq:m_t} and \eqref{eq:tilde m_t}, respectively, using the boundary value $f_{k,l} \in C(\TT^d; H^{s+1/2}(\partial\Omega))$ from \eqref{frequency-boundary-data}. As $\eps\rightarrow 0$
        \begin{align*}
            \Vert \tilde m_{k,l}(\dot {\boldsymbol D}, \dot {\boldsymbol A}, \tau_l) - m_{k,l}(\dot {\boldsymbol D}, \dot {\boldsymbol A})\Vert_{H^1(\Omega)} \leq o(1) \cdot \left(\Vert \dot {\boldsymbol D}\Vert_{L^2(\Omega; \RR^I)} + \Vert \dot {\boldsymbol A} \Vert_{H^1(\Omega; \RR^J)}\right),
        \end{align*}
        where the $o(1)$-decay depends on the choice of $\eta_\eps$. 
    \end{lemma}
    \begin{proof}
        For simplicity, we write $\tilde u_t(x) := \frac{u_{k,l}(x, t)}{\eta_\eps(t-\tau_l)}$ and $\tilde v_t(x) := \frac{v_{k,l}(x, t)}{\eta_\eps(t-\tau_l)}$, Then $L_t\tilde u_t = 0$ with $\tilde u_t\vert_{\partial\Omega} = g_{k,l}$ and $L_t \tilde v_t = \nabla\cdot(\dot D_t\nabla \tilde u_t) - \dot A_t\tilde u_t$ with $\tilde v_t \in H^1_0(\Omega)$. For $t, s\in \TT^d$, consider the following decomposition of $\tilde m_{k,l}(\dot {\boldsymbol D}, \dot {\boldsymbol A}, t) - \tilde m_{k,l}(\dot {\boldsymbol D}, \dot {\boldsymbol A}, s)$:
        \begin{align*}
            \tilde m_{k,l}(\dot {\boldsymbol D}, \dot {\boldsymbol A}, t) - \tilde m_{k,l}(\dot {\boldsymbol D}, \dot {\boldsymbol A}, s) = (\dot A_t - \dot A_s)\tilde u_t + \dot A_s(\tilde u_t - \tilde u_s) + (A_t - A_s)\tilde v_t + A_s (\tilde v_t - \tilde v_s)
        \end{align*}
        Let $I_1 := \Vert(\dot A_t - \dot A_s)\tilde u_t + \dot A_s(\tilde u_t - \tilde u_s)\Vert_{H^1(\Omega)}$ and $I_2 := \Vert(A_t - A_s)\tilde v_t + A_s (\tilde v_t - \tilde v_s)\Vert_{H^1(\Omega)}$. Since the functions $p_i$ and $q_j$ are continuous on $\TT^d$, let $\omega : \RR_+ \rightarrow \RR_+$ denote a common modulus of continuity; that is, let $\omega$ satisfy $\vert p_i(t) - p_i(s)\vert \leq \omega(\vert t-s\vert)$ and $\vert q_j(t) - q_j(s)\vert \leq \omega(\vert t-s\vert)$ for all $i, j$ and $t,s\in \TT^d$. To bound the first term $I_1$, we use
        \begin{align*}
            I_1 &\leq C\Vert \dot A_t - \dot A_s\Vert_{H^1(\Omega)}\Vert \tilde u_t\Vert_{W^{1,\infty}(\Omega)} + C\Vert\dot A_s\Vert_{H^1(\Omega)}\Vert \tilde u_t - \tilde u_s\Vert_{W^{1,\infty}(\Omega)}\\
            &\leq C\Vert \dot A_t - \dot A_s\Vert_{H^1(\Omega)}\Vert \tilde u_t\Vert_{W^{1,\infty}(\Omega)} + C\Vert\dot A_s\Vert_{H^1(\Omega)}(\Vert D_t - D_s\Vert_{H^{s+2}(\Omega)} + \Vert A_t - A_s\Vert_{H^{s+2}(\Omega)})
        \end{align*}
        where the second line follows from \Cref{solution-stability-2}. Observe that 
        \begin{equation}\label{eq:D_t-continuity}
            \Vert D_t - D_s\Vert_{H^{s+2}(\Omega)} \leq \sum_{i=1}^I \vert p_i(t) - p_i(s)\vert \Vert D_i\Vert_{H^{s+2}(\Omega)} \leq C\omega(\vert t-s\vert)\Vert {\boldsymbol D}\Vert_{H^{s+2}(\Omega; \RR^I)},
        \end{equation}
        where $\Vert {\boldsymbol D}\Vert_{H^{s+2}(\Omega; \RR^I)}$ is bounded by the preconditioning constant $c_2$,
        and estimates on $\dot A_t - \dot A_s$ and $A_t - A_s$ can be done similarly by
        \begin{align*}
            \Vert \dot A_t - \dot A_s\Vert_{H^1(\Omega)} &\leq \sum_{j=1}^J \vert q_j(t) - q_j(s)\vert \Vert \dot A_j\Vert_{H^1(\Omega)} \leq C\omega (\vert t - s\vert) \Vert \dot {\boldsymbol A}\Vert_{H^1(\Omega; \RR^J)},\\
            \Vert A_t - A_s\Vert_{H^{s+2}(\Omega)} &\leq \sum_{j=1}^J \vert q_j(t) - q_j(s)\vert \Vert A_j\Vert_{H^{s+2}(\Omega)} \leq C\omega (\vert t - s\vert) \Vert {\boldsymbol A}\Vert_{H^{s+2}(\Omega; \RR^J)}.
        \end{align*}
        The term $\Vert \tilde u_t \Vert_{W^{1,\infty}(\Omega)}$ can be estimated with \Cref{solution-stability-1}. Hence we have 
        \begin{align*}
            I_1 \leq C\omega(\vert t-s\vert)\Vert \dot {\boldsymbol A}\Vert_{H^1(\Omega; \RR^J)}.
        \end{align*}
        To bound the second term $I_2$, 
        \begin{align*}
            I_2 &\leq C(\Vert A_t - A_s\Vert_{W^{1,\infty}(\Omega)}\Vert\tilde v_t\Vert_{H^1(\Omega)} + \Vert A_s\Vert_{W^{1,\infty}(\Omega)}\Vert \tilde v_t - \tilde v_s\Vert_{H^1(\Omega)})\\
            &\leq C(\Vert A_t - A_s\Vert_{H^{s+2}(\Omega)}\Vert\tilde v_t\Vert_{H^1(\Omega)} + \Vert A_s\Vert_{H^{s+2}(\Omega)}\Vert \tilde v_t - \tilde v_s\Vert_{H^1(\Omega)}),
        \end{align*}
        The first summand can be estimated by the same reasoning as above in \eqref{eq:D_t-continuity}, To bound the second term $\Vert\tilde v_t - \tilde v_s\Vert_{H^1(\Omega)}$
        \begin{align*}
            \tilde v_t - \tilde v_s = G_{D_t, A_t}(\nabla\cdot (\dot D_t \nabla \tilde u_t) - \dot A_t \tilde u_t) - G_{D_s, A_s}(\nabla\cdot (\dot D_s \nabla \tilde u_s) - \dot A_s \tilde u_s),
        \end{align*}
        where an argument similar to that in \Cref{linear-perturbation-analysis} suffices. Let $g_t := \nabla\cdot (\dot D_t \nabla \tilde u_t) - \dot A_t \tilde u_t \in H^{-1}(\Omega)$. Then $\tilde v_t - \tilde v_s = G_{D_t, A_t} g_t - G_{D_s, A_s} g_s$, and we have the estimate
        \begin{align*}
            \Vert \tilde v_t - \tilde v_s\Vert_{H^1(\Omega)} &\leq \Vert (G_{D_t, A_t} - G_{D_s, A_s})g_t \Vert_{H^1(\Omega)} + \Vert G_{D_s, A_s}(g_t - g_s)\Vert_{H^1(\Omega)}\\
            &\leq \Vert G_{D_t, A_t} - G_{D_s, A_s}\Vert_{B(H^{-1}(\Omega), H^1(\Omega))} \Vert g_t \Vert_{H^{-1}(\Omega)}\\
            &\qquad + \Vert G_{D_s, A_s}\Vert_{B(H^{-1}(\Omega), H^1(\Omega))}\Vert g_t - g_s\Vert_{H^{-1}(\Omega)}.
        \end{align*}
        For the comparison of $g_t$ and $g_s$, we have
        \begin{align*}
            \Vert g_t - g_s\Vert_{H^{-1}(\Omega)} &\leq C(\Vert\dot D_t \nabla \tilde u_t - \dot D_s \nabla \tilde u_s\Vert_{L^2(\Omega)} + \Vert \dot A_t\tilde u_t - \dot A_s \tilde u_s\Vert_{L^2(\Omega)})\\
            &\leq C\Vert \tilde u_t\Vert_{W^{1,\infty}(\Omega)}(\Vert \dot D_t - \dot D_s\Vert_{L^2(\Omega)} + \Vert \dot A_t - \dot A_s\Vert_{L^2(\Omega)})\\
            &\qquad + C\Vert \tilde u_t - \tilde u_s\Vert_{W^{1,\infty}(\Omega)}(\Vert \dot D_s\Vert_{L^2(\Omega)} + \Vert \dot A_s\Vert_{L^2(\Omega)})\\
            &\leq C\omega(\vert t-s\vert) (\Vert \dot {\boldsymbol D}\Vert_{L^2(\Omega; \RR^I)} + \Vert \dot {\boldsymbol A}\Vert_{H^1(\Omega; \RR^J)})
        \end{align*}
        where the last line follows from \Cref{solution-stability-2}, the estimate \eqref{eq:D_t-continuity}, and similar estimates on $A_t - A_s$, $\dot A_t - \dot A_s$, and $\dot D_t - \dot D_s$ as in \eqref{eq:D_t-continuity}. The difference $G_{D_t, A_t} - G_{D_s, A_s}$ can be estimated with \eqref{eq:Gdiff}, and uniform estimates on $G_{D_s, A_s}$ and $g_t$ follow from \Cref{uniform-estimate-2}.
        Hence 
        \begin{align*}
            I_2 \leq C\omega(\vert t-s\vert)(\Vert \dot {\boldsymbol D}\Vert_{L^2(\Omega; \RR^I)} + \Vert \dot {\boldsymbol A}\Vert_{H^1(\Omega; \RR^J)}).
        \end{align*}
        Therefore, $\Vert \tilde m_{k,l}(\dot {\boldsymbol D}, \dot {\boldsymbol A}, t) - \tilde m_{k,l}(\dot {\boldsymbol D}, \dot {\boldsymbol A}, s)\Vert_{H^1(\Omega)} \leq C\omega(\vert t-s\vert)(\Vert \dot {\boldsymbol D}\Vert_{L^2(\Omega; \RR^I)} + \Vert \dot {\boldsymbol A}\Vert_{H^1(\Omega; \RR^J)})$, hence using this bound and the relation \eqref{eq:m_kl} we have
        \begin{align*}
            \Vert \tilde m_{k,l}(\dot {\boldsymbol D}, \dot {\boldsymbol A}, \tau_l) - m_{k,l}(\dot {\boldsymbol D}, \dot {\boldsymbol A})\Vert_{H^1(\Omega)} &\leq \int_{\TT^d} \Vert \tilde m_{k,l}(\dot {\boldsymbol D}, \dot {\boldsymbol A}, \tau_l) - \tilde m_{k,l}(\dot {\boldsymbol D}, \dot {\boldsymbol A}, t)\Vert_{H^1(\Omega)} \eta_\eps(t - \tau_l) dt\\
            &\leq C\int_{\TT^d}\omega(\vert t - \tau_l\vert) \left(\Vert \dot {\boldsymbol D}\Vert_{L^2(\Omega; \RR^I)} + \Vert \dot {\boldsymbol A} \Vert_{H^1(\Omega; \RR^J)}\right) \eta_\eps(t-\tau_l)dt\\
            &= o(1) \cdot \left(\Vert \dot {\boldsymbol D}\Vert_{L^2(\Omega; \RR^I)} + \Vert \dot {\boldsymbol A} \Vert_{H^1(\Omega; \RR^J)}\right),
        \end{align*}
        as $\eps \rightarrow 0$ by properties of the approximate identity and the modulus of continuity, i.e., $\omega(\eps) \in o(1)$ as $\eps\rightarrow 0$, which proves the claim.
    \end{proof}
    Let ${\boldsymbol M} := (M_{k,l})_{k,l}$ and ${\boldsymbol m} := (m_{k,l})_{k,l}$ be the systems of measurement operators and linearized measurement operators for $1\leq l\leq I+J$ and $1\leq k\leq 2n$ corresponding to boundary data of the form \eqref{frequency-boundary-data}, where ${\boldsymbol m}$ satisfies the mapping property of
    \begin{align*}
        {\boldsymbol m} := (m_{k,l})_{k,l} : L^2(\Omega'; \RR^I) \oplus H^1_0(\Omega'; \RR^J) \rightarrow H^1(\Omega; \RR^{2n(I+J)})
    \end{align*}
    as shown in \Cref{linear-forward-stability}.
    Now, we are ready to prove the stability for the linearized inverse problem in the frequency QPAT model stated in \Cref{main-3}:
    \begin{proof}[Proof of \Cref{main-3}]
        By \Cref{approximate-identity}, we have the estimate \begin{align*}
            \sum_{l,k} \Vert \tilde m_{k,l}(\dot {\boldsymbol D}, \dot {\boldsymbol A}, \tau_l) - m_{k,l}(\dot {\boldsymbol D}, \dot {\boldsymbol A})\Vert_{H^1(\Omega)} \leq o(1) \cdot \left(\Vert \dot {\boldsymbol D}\Vert_{L^2(\Omega; \RR^I)} + \Vert \dot {\boldsymbol A} \Vert_{H^1(\Omega; \RR^J)}\right).
        \end{align*}
        Then, for some constant $C>0$, as in \eqref{eq:pointwise-inverse}, we have
        \begin{align*}
            \Vert \dot {\boldsymbol D}\Vert_{L^2(\Omega; \RR^I)} + \Vert \dot {\boldsymbol A} \Vert_{H^1(\Omega; \RR^J)} &\leq C\sum_{l=1}^{I+J} (\Vert\dot D_{\tau_l}\Vert_{L^2(\Omega)} + \Vert\dot A_{\tau_l}\Vert_{H^1(\Omega)})\\
            &\leq C\sum_{l=1}^{I+J}\Vert (\tilde m_{k,l}(\dot {\boldsymbol D}, \dot {\boldsymbol A}, \tau_l) )_{k=1}^{2n}\Vert_{H^1(\Omega; \RR^{2n})}\\
            &\leq C\Vert {\boldsymbol m}(\dot {\boldsymbol D}, \dot {\boldsymbol A}) \Vert_{H^1(\Omega; \RR^{2n(I+J)})} + C o(1) \left(\Vert \dot {\boldsymbol D}\Vert_{L^2(\Omega; \RR^I)} + \Vert \dot {\boldsymbol A} \Vert_{H^1(\Omega; \RR^J)}\right).
        \end{align*}
        Choosing $\eps>0$ sufficiently small gives the stability estimate, i.e., there exists a constant $C>0$ such that the system of maps ${\boldsymbol m}$ satisfies the bi-Lipschitz estimate: 
        \begin{align*}
            C^{-1}\Vert {\boldsymbol m}(\dot {\boldsymbol D}, \dot {\boldsymbol A})\Vert_{H^1(\Omega; \RR^{2n(I+J)})} \leq \Vert\dot {\boldsymbol D}\Vert_{L^2(\Omega; \RR^I)} + \Vert\dot {\boldsymbol A}\Vert_{H^1(\Omega; \RR^J)} \leq C\Vert {\boldsymbol m}(\dot {\boldsymbol D}, \dot {\boldsymbol A})\Vert_{H^1(\Omega; \RR^{2n(I+J)})}.
        \end{align*}
        This concludes the proof of \Cref{main-3}.
    \end{proof}
    
    \section{Nonlinear Local Inverse Stability}\label{section-6}

    In this section, we establish the local nonlinear inverse stability result \Cref{main-2} by applying a Banach space nonlinear inverse stability theorem \cite{Stefanov2009} together with the linear analysis developed in the previous sections. 
    We first present the nonlinear stability theorem whose proof is provided in Appendix~\ref{app:banach}:
    \begin{proposition}[{\cite[Theorem 2]{Stefanov2009}}]\label{inverse-function-theorem}
        Let $X'' \subseteq X \subseteq X'$ and $Y'' \subseteq Y' \subseteq Y$ be Banach spaces such that the following interpolation inequalities hold for some constant $C_1 > 0$
        \begin{equation} \label{eq:interpolation}
            \Vert x\Vert_{X}\leq C_1 \Vert x\Vert_{X'}^{\mu_1}\Vert x\Vert_{X''}^{1-\mu_1}\quad\text{and}\quad \Vert y\Vert_{Y'} \leq C_1 \Vert y\Vert_{Y}^{\mu_2}\Vert y\Vert_{Y''}^{1-\mu_2}\quad\text{with}\;\alpha = \mu_1\mu_2 > 1/2,
        \end{equation}
        Given $x_0\in X''$ and an open neighborhood $U\subseteq X''$ of $x_0$, let $F : U\rightarrow Y$ be a function that satisfies the second-order expansion $F(x_2) = F(x_1) + \delta F_{x_1}(x_2-x_1) + R_{x_1}(x_2-x_1)$ for all $x_1, x_2\in U$ with the following uniform estimates for all $x\in U$ and $x + h\in U$
        \begin{equation}\label{forward-regularity}
            \begin{aligned}
                \Vert\delta F_x (h) \Vert_{Y''}&\leq C_2 \Vert h \Vert_{X''} \quad &\text{for some} \; C_2>0
                \;\text{independent of}\;x\in U,\\
                \Vert R_x(h)\Vert_{Y} &\leq C_2 \Vert h\Vert^2_X \quad &\text{for some} \; C_2>0
                \;\text{independent of}\;x\in U,\\
                \Vert h\Vert_{X'} &\leq C_3\Vert\delta F_{x} (h)\Vert_{Y'}\quad &\text{for some} \; C_3>0\;\text{independent of}\;x\in U,
            \end{aligned}
        \end{equation}
        then there exists an open neighborhood $V\subseteq U$ of $x_0$ in $X''$
        such that for all $x, y\in V$, the following H\"older-type local inverse stability holds for $F$ on $V$:
        \begin{equation}
            \norm{y - x}_X\leq C \norm{F(y) - F(x)}_Y^\alpha\quad\text{for some}\quad C>0\quad\text{for all}\quad x,y\in V.
        \end{equation}
    \end{proposition}
    
    We begin by choosing appropriate Banach spaces on which to pose the nonlinear inverse problem using the preceding linear analysis.
    \subsection{Functional Setup}\label{functional-setup}
    We now choose the appropriate Banach spaces in \Cref{inverse-function-theorem} for our context. Let $s > n/2$
    and make the following choices of spaces $X$, $X'$, $X''$, and $Y$, $Y'$, $Y''$:
    \begin{align*}
        X'' &= \{\varphi \in W^{s+2, \infty}(\Omega; \RR^{I+J}) : \operatorname{supp}\varphi_k \subseteq \overline{\Omega'} \;\text{for all}\; 1\leq k\leq I+J\};\\ 
        X'\phantom{'} &= L^2(\Omega'; \RR^I)\oplus H^1_0(\Omega'; \RR^J);\\ 
        X\phantom{''} &= L^\infty(\Omega'; \RR^{I+J}) \cap X'; \\
        Y'' &= H^{s+1}(\Omega; \RR^{2n(I+J)});\\ 
        Y'\phantom{'} &= H^1(\Omega; \RR^{2n(I+J)});\\ 
        Y\phantom{''} &= L^2(\Omega; \RR^{2n(I+J)}).
    \end{align*}
    The norms on $X'$ and $Y, Y', Y''$ are their canonical norms, while the norm on $X''$ is the $W^{s+2,\infty}(\Omega'; \RR^{I+J})$ norm. The norm on $X$ is defined as follows: for all $\varphi = (\zeta_1,\ldots, \zeta_I, \psi_1, \ldots, \psi_J)\in X$, $\Vert \varphi\Vert_{X}$ is given by
    \begin{align*}
        \Vert \varphi\Vert_X^2 := \Vert\zeta\Vert_{L^\infty(\Omega; \RR^I)}^2 + \Vert\zeta\Vert_{L^2(\Omega; \RR^I)}^2 + \Vert\psi\Vert_{L^\infty(\Omega; \RR^J)}^2 + \Vert\psi\Vert_{H^1(\Omega; \RR^J)}^2 = \Vert\varphi\Vert_{L^\infty(\Omega; \RR^{I+J})}^2 + \Vert\varphi\Vert_{X'}^2.
    \end{align*}
    
    We first verify the interpolation results in \eqref{eq:interpolation}. The spaces $X, X', X''$ describe the regularity of the unknown diffusion and attenuation coefficients.
    To determine the interpolation exponent $\mu_1$ among $X$, $X'$, and $X''$, it suffices to interpolate $L^\infty(\Omega)$ between $L^2(\Omega)$ and $W^{s+2,\infty}(\Omega)$.
    Applying \Cref{sobolev-interpolation-1}, we see that for all $\varphi \in W^{s+2,\infty}(\Omega)$:
    \begin{align*}
        \Vert \varphi\Vert_{L^\infty(\Omega)} \leq C\Vert \varphi\Vert^{\mu_1}_{L^2(\Omega)}\Vert \varphi\Vert_{W^{s+2,\infty}(\Omega)}^{1-\mu_1} \qquad \text{ where } \mu_1 := \tfrac{2s + 4}{n + 2s + 4}.
    \end{align*}
    The spaces $Y, Y', Y''$ will be used to model a system of $2n(I + J)$ interior measurements. By standard results on Sobolev interpolation between $L^2$-type Sobolev spaces (refer to \Cref{sobolev-interpolation-2}), we have for all $\varphi\in H^{s+1}(\Omega)$:
    \begin{align*}
        \Vert\varphi\Vert_{H^1(\Omega)} \leq C\Vert \varphi\Vert_{L^2(\Omega)}^{\mu_2} \Vert \varphi \Vert_{H^{s+1}(\Omega)}^{1-\mu_2}, \qquad \text{ where } \mu_2 := \tfrac{s}{s+1}.
    \end{align*}
    Hence, with these choices, we have
    \begin{align*}
        \Vert x \Vert_X &\leq C\Vert x \Vert_{X'}^{\mu_1}\Vert x\Vert_{X''}^{1-\mu_1}\quad \text{for all}\quad x \in X'';\\
        \Vert y \Vert_{Y'} &\leq C\Vert y\Vert_Y^{\mu_2}\Vert y \Vert_{Y''}^{1-\mu_2}\quad\text{for all}\quad y\in Y'';
    \end{align*}
    where the first interpolation follows from
    \begin{align*}
        \Vert x\Vert_{X}^2 = \Vert x\Vert_{L^\infty(\Omega; \RR^{I+J})}^2 + \Vert x \Vert_{X'}^2 \leq C(\Vert x\Vert_{X'}^{\mu_1}\Vert x\Vert_{X''}^{1-\mu_1})^2 + \Vert x\Vert_{X'}^2 \leq (C+1)(\Vert x\Vert_{X'}^{\mu_1}\Vert x\Vert_{X''}^{1-\mu_1})^2.
    \end{align*}
    Let $\alpha := \mu_1\mu_2$. As in \Cref{inverse-function-theorem}, to have $\alpha > 1/2$ requires us to pick $s > (n/4 - 1/2) + \sqrt{(n/4+1/2)^2 + 2} = n/2 + O(1/n)$. While restricting to integer $s$, for any $\alpha \in (1/2, 1)$, there exists an integer $s > n/2$ such that $\mu_1\mu_2 \geq \alpha$. 
    Indeed, after some algebra, we see that $s/n$ satisfies the asymptotic $s/n = \beta/(2(1-\beta)) + O(1/n)$ with $\beta = \mu_1\mu_2$. Then $\alpha$-H\"older stability can be established for all $\alpha > 1/2$ with integer $s$ in \Cref{inverse-function-theorem}.
    \subsection{Inverse Stability via Banach Inverse Theorem}
    In this section, we verify the rest of the conditions in \Cref{inverse-function-theorem} and conclude the proof of \Cref{main-2}. 
    \begin{proof}[Proof of~\Cref{main-2}]

        For the frequency-averaged QPAT problem \eqref{frequency-QPAT}, apply \Cref{inverse-function-theorem} with $F$ given by ${\boldsymbol M}({\boldsymbol D}, {\boldsymbol A}) = {\boldsymbol{\mathfrak m}}$, where ${\boldsymbol M} = (M_{k,l})_{1\leq k\leq 2n, 1\leq l \leq I + J}$ is the system of measurement operators.  The neighborhood $U$ corresponds to the preconditioned set $\mathcal A_\infty^I \times \mathcal A_\infty^J$. We now verify that the hypotheses of \Cref{inverse-function-theorem} are satisfied: 
        \begin{enumerate}[leftmargin=*, label = (\alph*)]
            \item Since the diffusion and attenuation coefficients are preconditioned to be one on $\Omega - \Omega'$, we take the neighborhood $U = \mathcal A_\infty^I \times \mathcal A_\infty^J$ and let $\tilde U$ be the affine shift
            \begin{align*}
                \tilde U = \{x \in X'': x_k + 1 \in \mathcal A_\infty\;\text{for all}\; 1\leq k \leq I + J\},
            \end{align*}
            We introduce a shift by one since $D_i$ and $A_j$ are a priori assumed to be constant one near the boundary of $\Omega$.
            Since the pointwise evaluation map is continuous under the $X''$ norm, $\tilde U$ is open in $X''$ by the definition of $\mathcal A_\infty$. And $F$ here is given by
            \begin{align*}
                F : \tilde U \rightarrow Y,\; x\mapsto {\boldsymbol M}(\boldsymbol D, \boldsymbol A),\;\text{where}\; D_i = x_i + 1, A_j = x_{I+j} + 1\;\text{for all}\; 1 \leq i \leq I, 1\leq j\leq J.
            \end{align*}
            \item We now check the conditions in \eqref{forward-regularity}. The first condition is given by \Cref{linear-forward-estimate}:
            \begin{align*}
                \Vert\delta F (h) \Vert_{Y''} = \Vert {\boldsymbol m}(\delta {\boldsymbol D}, \delta {\boldsymbol A}) \Vert_{H^{s+1}(\Omega; \RR^{2n(I+J)})} \leq C \Vert (\delta {\boldsymbol D}, \delta {\boldsymbol A})\Vert_{W^{s+2,\infty}(\Omega; \RR^{I+J})} = C \Vert h \Vert_{X''}.
            \end{align*}
            The constant can be bounded uniformly for all base points in the admissible set, which proves the first required estimate.
            \item For the second estimate in \eqref{forward-regularity}, we have
            \begin{align*}
                \Vert R(h)\Vert_{Y} = \Vert R(\delta {\boldsymbol D}, \delta {\boldsymbol A})\Vert_{L^2(\Omega; \RR^{2n(I+J)})} \leq C\Vert(\delta {\boldsymbol D}, \delta {\boldsymbol A})\Vert_{L^\infty(\Omega; \RR^{I+J})}^2 = C\Vert h\Vert^2_X,
            \end{align*}
            which is given by \Cref{analyticity} and holds after possibly restricting the open neighborhood $U$ with the $L^\infty(\Omega; \RR^{I+J})$ norm (i.e., in the topology of $X$). Likewise, the constant $C>0$ can be chosen uniformly for all base points in the admissible set, see \Cref{analyticity}.
            \item 
            To show the last inequality in \eqref{forward-regularity}, we have
            \begin{align*}
                \Vert h\Vert_{X'} = \Vert(\delta {\boldsymbol D}, \delta {\boldsymbol A})\Vert_{L^2(\Omega; \RR^I) \oplus H^1(\Omega; \RR^J)} \leq C \Vert {\boldsymbol m}(\delta {\boldsymbol D}, \delta {\boldsymbol A})\Vert_{H^1(\Omega; \RR^{2n(I+J)})} = C\Vert\delta F (h)\Vert_{Y'},
            \end{align*}
            This follows from \Cref{main-3}; by \Cref{perturbation-remark}, the constant $C > 0$ can be chosen uniformly on $\tilde U$.
            \item
            In summary, we have shown that, given any $\alpha > 1/2$, there exists $s > n/2$ such that the interpolation
            \begin{align*}
                \Vert x \Vert_X &\leq C\Vert x \Vert_{X'}^{\mu_1}\Vert x\Vert_{X''}^{1-\mu_1}\quad \text{for all}\quad x \in X'',\\
                \Vert y \Vert_{Y'} &\leq C\Vert y\Vert_Y^{\mu_2}\Vert y \Vert_{Y''}^{1-\mu_2}\quad\text{for all}\quad y\in Y''
            \end{align*}
            holds with some $\mu_1$ and $\mu_2$ such that $\mu_1\mu_2 = \alpha > 1/2$. Given any $({\boldsymbol D}, {\boldsymbol A})\in \mathcal A_\infty^I\times \mathcal A_\infty^J$, there exists an open subset $\tilde U$ in $X''$, given by an affine shift of an open neighborhood $U$ around $({\boldsymbol D}, {\boldsymbol A})$, on which the estimates \eqref{forward-regularity} hold:
            \begin{equation*}
                \Vert\delta F_x (h) \Vert_{Y''}\leq C \Vert h \Vert_{X''},\quad
                \Vert R_x(h)\Vert_{Y} \leq C \Vert h\Vert^2_X,\quad
                \Vert h\Vert_{X'} \leq C\Vert\delta F_{x} (h)\Vert_{Y'},
            \end{equation*}
            for some constant $C>0$ independent of $x \in \tilde U$.
            \item 
            Now applying \Cref{inverse-function-theorem}, we have the nonlinear local inverse stability result:
            \par
            There are $2n(I+J)$ measurements determined by $C(\TT^d; H^{s+1/2}(\partial\Omega))$ boundary data $f_{1,1}, \ldots, f_{2n, I+J}$ such that the following H\"older-type inverse stability estimate holds:
            \begin{align*}
                \Vert({\boldsymbol D}_1, {\boldsymbol A}_1) - ({\boldsymbol D}_2, {\boldsymbol A}_2)\Vert_{L^\infty(\Omega; \RR^{I+J})} \leq C_\alpha \Vert {\boldsymbol M}({\boldsymbol D}_1, {\boldsymbol A}_1)  - {\boldsymbol M}({\boldsymbol D}_2, {\boldsymbol A}_2)\Vert_{{L}^2(\Omega; \RR^{2n(I+J)})}^\alpha,
            \end{align*}
            for all $({\boldsymbol D}_1, {\boldsymbol A}_1), ({\boldsymbol D}_2, {\boldsymbol A}_2)\in U$, where ${\boldsymbol M}:= (M_1, M_2,\ldots, M_{2n(I+J)})$ is the system of measurement operators defined in \eqref{eq:measop} for the $2n(I+J)$ boundary data $f_{k,l}(x, t)$.
        \end{enumerate}
        This concludes the proof of \Cref{main-2}.
    \end{proof}

\section*{Acknowledgments}
The research of Yang Yang is partially supported by the National Science Foundation (NSF) grants DMS-2237534 and DMS-2220373. Yunan Yang and Yunhao Sun are partially supported by NSF grants DMS-2409855 and DMS-2540324, and by the Office of Naval Research under Award No.~N00014-24-1-2088.

    \appendix
    \section{Appendix}
    \subsection{Elliptic Regularity Estimates}
    In this section, we collect some essential properties of the elliptic operator
    \begin{align*}
        Lu = -\nabla\cdot(D(x)\nabla u(x)) + A(x)u(x),
    \end{align*}
    where $D, A \in L^\infty(\Omega)$ are uniformly positive and bounded on $\Omega$. 
    Recall that the solution of the boundary value problem
    $Lu=0$, $u|_{\partial\Omega}=f$ admits the representation $u = Tf - GLTf$ where $T$ is the harmonic extension and $G$ the inverse of $L$ equipped with the homogeneous Dirichlet boundary condition, see~\eqref{eq:utrep}.
    \begin{lemma}\label{uniform-estimate-2}
        Suppose that $D, A \in L^\infty(\Omega)$ satisfy $k_1 < D, A < k_2$ and that $f\in H^{1/2}(\partial \Omega)$ satisfies $\Vert f\Vert_{H^{1/2}(\partial\Omega)} < k_3$ for some positive constants $k_1, k_2, k_3 > 0$. Then
        \begin{itemize}[leftmargin = *]
            \item the operator norm $\Vert L\Vert_{B(H^1(\Omega), H^{-1}(\Omega))}$ is bounded by $k_2$;
            \item if $G :H^{-1}(\Omega) \rightarrow H^1_0(\Omega)$ is the inverse of $L$, then the operator norm $\Vert G\Vert_{B(H^{-1}(\Omega), H^1(\Omega))}$ is bounded by $1/k_1$;
            \item if $u$ is the solution to $Lu = 0$, $u\vert_{\partial\Omega} = f$, then $\Vert u\Vert_{H^1(\Omega)}$ can be bounded by a constant $C>0$ depending only on $k_1, k_2$, and $k_3$.
        \end{itemize}
    \end{lemma}
    \begin{proof}
        Well-posedness of the solution and the solution operator follows from standard results \cite{Gilbarg2001}, so we only need to show the dependence on the constants.
        To show the bound on $L$, let $\varphi_1\in H^1_0(\Omega)$ and $\varphi_2\in H^1(\Omega)$ be arbitrary. Then
        \begin{align*}
            \int_\Omega \varphi_1 L\varphi_2 dx = \int_\Omega D\nabla\varphi_1 \cdot \nabla \varphi_2 + A\varphi_1\varphi_2 \leq k_2\Vert \varphi_1\Vert_{H^1(\Omega)}\Vert\varphi_2\Vert_{H^1(\Omega)},
        \end{align*}
        hence the operator norm of $L$ is bounded by $k_2$.
        \par
        To show the second claim, given $g\in H^{-1}(\Omega)$, let $v = Gg \in H^1_0(\Omega)$. Then
        \begin{align*}
            \Vert v\Vert_{H^1(\Omega)}^2  \leq \frac{1}{k_1}\int_\Omega D \vert\nabla v\vert^2 + Av^2 dx \leq \frac{1}{k_1}\int_\Omega v Lv dx \leq \frac{1}{k_1}\Vert g\Vert_{H^{-1}(\Omega)}\Vert v\Vert_{H^1(\Omega)},
        \end{align*}
        where the second integral is interpreted in the sense of distributions. This shows the second claim.
        \par
        To show the last claim, let $T : H^{1/2}(\partial\Omega)\rightarrow H^1(\Omega)$ be the harmonic extension operator. Then the solution to $Lu = 0$, $u\vert_{\partial\Omega} = f$ is given by $u = Tf - GLTf$. Hence
        \begin{align*}
            \Vert u\Vert_{H^1(\Omega)} &\leq \Vert Tf\Vert_{H^1(\Omega)} + \Vert GLTf\Vert_{H^1(\Omega)} \leq (1 + \Vert G\Vert_{B(H^{-1}(\Omega), H^1(\Omega))}\Vert L\Vert_{B(H^1(\Omega), H^{-1}(\Omega))})\Vert Tf\Vert_{H^1(\Omega)}\\
            &\leq \left(1 + k_2/k_1\right)\Vert T\Vert_{B(H^{1/2}(\partial \Omega), H^1(\Omega))} \|f\|_{H^{1/2}(\partial\Omega)} \\
            & \leq \left(1 + k_2/k_1\right)\Vert T\Vert_{B(H^{1/2}(\partial \Omega), H^1(\Omega))}k_3
        \end{align*}
        which proves the claim.
    \end{proof}

    Lemma~\ref{uniform-estimate-2} and its proof imply the following fact for the $t$-dependent operator $L_t$.
    
    \begin{corollary}\label{uniform-estimate}
        Consider the $t$-dependent boundary value problem in \eqref{frequency-QPAT} with $({\boldsymbol D}, {\boldsymbol A}) \in L^\infty(\Omega; \RR^{I+J})$. Suppose $k_1 < D_i, A_j < k_2$ for all $i, j$ for some positive constants $k_1, k_2$. Then the operator norms 
        $\Vert L_t\Vert_{B(H^1(\Omega), H^{-1}(\Omega))}$, $\Vert G_t\Vert_{B(H^{-1}(\Omega), H^1(\Omega))}$ are uniformly bounded for all $t\in \TT^d$ with a constant depending only on $k_1$ and $k_2$. Given a fixed boundary value $f\in C(\TT^d; H^{1/2}(\partial\Omega))$, we have
        \begin{align*}
            \sup\nolimits_{t\in \TT^d}\Vert u_t\Vert_{H^1(\Omega)} \leq C\sup\nolimits_{t\in \TT^d}\Vert f_t\Vert_{H^{1/2}(\partial\Omega)} = C \Vert f \Vert_{C(\TT^d; H^{1/2}(\partial\Omega))}
        \end{align*}
        where the constant $C=\left(1 + k_2/k_1\right)\Vert T\Vert_{B(H^{1/2}(\partial \Omega), H^1(\Omega))}$.
        In particular, the norm $\Vert u_t\Vert_{H^1(\Omega)}$ is bounded for all $t\in \TT^d$ with a constant that depends on $k_1, k_2, f$.
    \end{corollary}
    To obtain sharper stability estimates for the linearized QPAT map, we also establish $L^\infty$-type bounds for solutions $u$ to the Dirichlet problem $Lu = 0$, $u\vert_{\partial\Omega}=f$
    when $D,A \in H^{s+2}(\Omega)$ with $s>n/2$. In particular, higher regularity of $D$ and $A$ yields improved control on $u$, which we summarize in the following estimates.
    \begin{lemma}\label{solution-stability-1}
        Let $(D, A) \in \mathcal A\times\mathcal A$ be a pair of admissible coefficients, and let $f\in H^{s+1/2}(\partial\Omega)$ be the boundary condition. Then the solution $u$ is in $W^{1,\infty}(\Omega)$ and satisfies
        \begin{align*}
            \Vert u\Vert_{W^{1,\infty}(\Omega)} \leq C\Vert f \Vert_{H^{s + 1/2}(\partial\Omega)},
        \end{align*}
        where $C>0$ can be chosen uniformly for all $(D, A)\in \mathcal A\times\mathcal A$.
    \end{lemma}
    To show the above claim, we first recall the following known result:
    \begin{theorem}[{\cite[Theorem 9.15, Lemma 9.17]{Gilbarg2001}}]\label{GT9.15}
        Let $\Omega$ be a $C^{1,1}$ domain in $\RR^n$, and let the operator $L$ be strictly elliptic in $\Omega$ with $C(\overline\Omega)$ coefficients. Then if $g\in L^p(\Omega)$ with $1 < p < \infty$, the Dirichlet problem $Lu = g$, $u\vert_{\partial \Omega} = 0$ has a unique solution $u \in W^{2,p}(\Omega) \cap W^{1,p}_0(\Omega)$ satisfying
        \begin{align*}
            \Vert u \Vert_{W^{2,p}(\Omega)} \leq C\Vert Lu \Vert_{L^p(\Omega)}
        \end{align*}
        for some constant $C>0$.
    \end{theorem}
    \begin{remark}
        Note that the constant $C>0$ above depends only on the domain $\Omega$ and the ellipticity and boundedness of the operator $L$, since here we have assumed that $D, A\in \mathcal A$, the constant $C > 0$ can be taken to be independent of $D$ and $A$.
    \end{remark}
    \begin{proof}[Proof of \Cref{solution-stability-1}]
        Assume that $D, A\in H^{s+2}(\Omega)$ belong to the admissible set $\mathcal A$ with $s > n/2$. Then $L$ is strictly elliptic with continuous coefficients. For boundary data $f\in H^{s+1/2}(\partial \Omega)$, the harmonic extension satisfies $Tf\in H^{s+1}(\Omega)$. By Sobolev embedding, $Tf \in H^{s+1}(\Omega) \hookrightarrow W^{2,p}(\Omega)$ for $p = \frac{2n}{n - 2s + 2} > n$ \cite[2.3.3]{Edmunds1996}, and $LTf \in L^p(\Omega)$. Now consider the PDE given by $Lv = -LTf$, $v\vert_{\partial\Omega} = 0$. By \Cref{GT9.15}, $\Vert v\Vert_{W^{2,p}(\Omega)} \leq C\Vert LTf\Vert_{L^p(\Omega)}$. Let $u = Tf + v$. Then $Lu = 0$ with trace $f$, and we obtain the bound
        \begin{align*}
            \Vert u\Vert_{W^{2,p}(\Omega)} &\leq \Vert Tf \Vert_{W^{2,p}(\Omega)} + \Vert u - Tf \Vert_{W^{2,p}(\Omega)} = \Vert Tf \Vert_{W^{2,p}(\Omega)} + \Vert v \Vert_{W^{2,p}(\Omega)}\\
            &\leq \Vert Tf\Vert_{W^{2,p}(\Omega)} + C\Vert LTf \Vert_{L^p(\Omega)}\\
            &\leq \Vert Tf\Vert_{W^{2,p}(\Omega)} + C\Vert L\Vert_{B(W^{2,p}(\Omega), L^p(\Omega))} \Vert Tf \Vert_{W^{2,p}(\Omega)}\\
            &\leq C\Vert f \Vert_{H^{s+1/2}(\partial\Omega)}.
        \end{align*}
        Finally $u\in W^{2,p}(\Omega)\hookrightarrow W^{1,\infty}(\Omega)$ by Morrey's embedding.
    \end{proof}
    \begin{lemma}\label{solution-stability-2}
        Let $(D, A)$ and $(D^\ast, A^\ast)$ be two admissible pairs of parameters in $\mathcal A\times \mathcal A$, and let $f\in H^{s+1/2}(\partial\Omega)$ be a fixed boundary value. Denote the corresponding solutions by $u$ and $u^\ast$, respectively. Then
        \begin{align*}
            \Vert u -u^\ast\Vert_{W^{1,\infty}(\Omega)} \leq C\Vert(D, A) - (D^\ast, A^\ast)\Vert_{H^{s+2}(\Omega; \RR^2)},
        \end{align*}
        for a constant $C>0$.
    \end{lemma}
    \begin{proof}
        We let $w = u - u^\ast \in W_0^{1,p}(\Omega) \cap W^{2,p}(\Omega)$, $\delta D = D^\ast - D$ and $\delta A = A^\ast - A$. Then $w$ solves the following PDE
        \begin{align*}
            Lw = -\nabla\cdot(D(x)\nabla w(x)) + A(x)w(x) = - \nabla \cdot (\delta D(x) \nabla u^\ast(x)) + \delta A(x) u^\ast(x) =: g(x).
        \end{align*}
        By \Cref{GT9.15}, we again have
        \begin{align*}
            \Vert w\Vert_{W^{2,p}(\Omega)} \leq C \Vert g\Vert_{L^p(\Omega)} \leq C \Vert u^\ast\Vert_{W^{2, p}(\Omega)}(\Vert \delta D\Vert_{W^{1,\infty}(\Omega)} + \Vert \delta A\Vert_{L^\infty(\Omega)}).
        \end{align*}
        The term $\Vert u^\ast\Vert_{W^{2,p}(\Omega)}$ is bounded by \Cref{solution-stability-1}.
        Again, by Sobolev embedding and Morrey's inequality, $\Vert w\Vert_{W^{1,\infty}(\Omega)} \leq C\Vert (\delta D, \delta A)\Vert_{H^{s+2}(\Omega; \RR^2)}$ for some constant $C > 0$.
    \end{proof}
    Now assume that $D, A\in W^{s+2, \infty}(\Omega)$ with $s\in \NN$. We show that the solution satisfies regularity $u \in H^{s+1}(\Omega)$.
    \begin{lemma}\label{Schauder-estimate}
        Let $(D, A) \in \mathcal A_\infty \times\mathcal A_\infty$ be a pair of admissible coefficients, let $f\in H^{s+1/2}(\partial\Omega)$ be the boundary value, and let $g\in H^{s-1}(\Omega)$. Then the solution $u$ to $Lu = g$, $u\vert_{\partial\Omega} = f$ belongs to $H^{s+1}(\Omega)$ and satisfies the norm estimate
        \begin{align*}
            \Vert u \Vert_{H^{s+1}(\Omega)} \leq C(\Vert g \Vert_{H^{s-1}(\Omega)} + \Vert f\Vert_{H^{s+1/2}(\partial\Omega)}),
        \end{align*}
        where the constant $C>0$ can be chosen uniformly for all $(D, A)\in \mathcal A_\infty \times\mathcal A_\infty$.
    \end{lemma}
    \begin{proof}
        We bootstrap the regularity of $u$. By standard elliptic theory, we have $u\in H^1(\Omega)$. Assuming that $u$ has $H^t(\Omega)$ regularity for $1\leq t\leq s$, applying $L^\infty$ control on the derivatives of the coefficients, we have
        \begin{align*}
            - \Delta u(x) = \frac{\nabla D(x)}{D(x)} \cdot \nabla u(x) - \frac{A(x)}{D(x)}u(x) + \frac{1}{D(x)}g(x) \in H^{t-1}(\Omega),\quad  u\vert_{\partial\Omega} \in H^{s+1/2}(\partial\Omega),
        \end{align*}
        Denoting the right-hand side by $B$, we obtain
        \begin{align*}
            \Vert B\Vert_{H^{t-1}(\Omega)} \leq C\left(\Vert\nabla D/D\Vert_{W^{t-1,\infty}(\Omega)} + \Vert A/D\Vert_{W^{t-1,\infty}(\Omega)}\right)\Vert u\Vert_{H^t(\Omega)} + C\Vert 1/D\Vert_{W^{t-1,\infty}(\Omega)}\Vert g\Vert_{H^{t-1}(\Omega)},
        \end{align*}
        for some constant $C>0$.
        Inverting the Laplacian gives $u\in H^{t+1}(\Omega)$. Note that all the terms in the above equation involving $D$ and $A$ can be bounded uniformly for all $D, A\in \mathcal A_\infty$. This proves the claim.
    \end{proof}
    \subsection{Complex Geometric Optics Solutions}
    In this section, we recall some essential results on CGO solutions to elliptic PDEs, again, we start with the equation $Lu = -\nabla\cdot(D(x)\nabla u(x)) + A(x)u(x) = 0$ and assume $D$ and $A$ are uniformly positive $H^{s+2}(\Omega)$ functions with $s > n/2$.
 
    \subsubsection{Liouville Transform and CGO Solutions}
    For the Liouville transformation of the above elliptic PDE, define the potential $V$ from the conductivity and attenuation coefficients by
    \begin{align*}
        V := \Delta\sqrt{D}/\sqrt{D} + A/D.
    \end{align*}
    A function $u$ on $\Omega$ solves $Lu = 0$ if and only if $w = \sqrt{D}u$ solves the Schr\"odinger equation
    \begin{align*}
        (-\Delta + V)w = 0.
    \end{align*}
    $w$ is called a CGO solution if $w = e^{\xi\cdot x}(1 + \psi(x))$ with $\xi\in\CC^n$, $\xi\cdot\xi = 0$, where $\psi$ solves the PDE
    \begin{equation}\label{cgo}
        -\Delta \psi(x) - 2\xi\cdot \nabla \psi(x) + V(x)\psi(x) = -V(x)\;\;\text{on}\;\; \Omega.
    \end{equation}
    We recall the following classical result on CGO regularity:
    \begin{theorem}[{\cite[corollary 2.5]{Sylvester1987}; \cite[corollary 3.2]{Stefanov2009}}]\label{cgo-1}
        Let $s > n/2$ and $\gamma > 0$. Then there exists a constant $c = c(\Omega, \gamma) > 0$ such that, if $\xi \in \CC^n$ and $V \in H^s(\Omega)$ satisfy $\xi\cdot \xi = 0$, $\vert\xi\vert \geq \gamma$, and $\Vert V\Vert_{H^s(\Omega)} \leq c\vert\xi\vert$, there exists a solution $\psi \in H^{s+1}(\Omega)$ to \eqref{cgo} that satisfies
        \begin{align*}
            \vert\xi\vert \Vert \psi \Vert_{H^s(\Omega)} + \Vert \psi\Vert_{H^{s+1}(\Omega)} \leq C\Vert V \Vert_{H^s(\Omega)}
        \end{align*}
        for some constant $C > 0$ independent of $\xi$.
    \end{theorem}
    We make a few comments on the choice of the Sobolev exponent $s$ as well as the choice of the preconditioning for the QPAT problem.
    Given that $D, A\in H^{s+2}(\Omega)$ for $s > n/2$, we see that $V \in H^s(\Omega)$,
    hence $\psi, w \in H^{s+1}(\Omega)$.
    \par
    Taking $D, A\in W^{s + 2, \infty}(\Omega)$ with $s > n/2$, the conclusion of the preceding theorem also holds by replacing the $H^s(\Omega)$ norm on $V$ with the $W^{s, \infty}(\Omega)$ norm, which validates our choice of the admissible set $\mathcal A_\infty$.
    
    \subsubsection{Matrix Field Generated by CGO Solutions}
    In this section, we recall the construction of the $2n$ real CGO solutions in \cite{Bal2010} for the nonlinear QPAT problem. For a CGO solution $w = e^{\xi\cdot x}(1 + \psi(x))$ to the Schr\"odinger equation, we have $\nabla w = e^{\xi\cdot x}(\xi + \xi\psi(x) + \nabla\psi(x))$,
    if the last two summands on the right-hand side are small, then the vector field $\nabla w$ is controlled by $\xi$. This argument is formalized in the following theorem.
    \begin{theorem}[{\cite[theorem 3.4]{Bal2010}}]\label{cgo-2}
        Let $s > n/2$, and consider two solutions $w_{i} = e^{\xi_i \cdot x}(1 + \psi_i(x))$ for $i = 1,2$ to the Schr\"odinger equation $(-\Delta + V)w = 0$, suppose that $r = \vert\xi_1\vert = \vert \xi_2\vert$ satisfies the lower bound $r > \gamma$, $\Vert V\Vert_{H^s(\Omega)} \leq c(\Omega, \gamma)r$, then
        \begin{align*}
            \frac{1}{r} e^{-(\xi_1 + \xi_2)\cdot x}(w_1\nabla w_2 - w_2 \nabla w_1) = \frac{\xi_2 - \xi_1}{r} + h(x),\;\; \text{where}\;\; \Vert h\Vert_{H^s(\Omega; \CC^n)} < \frac{C}{r}
        \end{align*}
        for some constant $C > 0$ independent of $\xi_i$'s.
    \end{theorem}
    The preceding estimate shows that, for sufficiently large complex phases, the vector field $w_1\nabla w_2 - w_2\nabla w_1$ is uniformly close to the constant direction $\xi_2-\xi_1$. One can find such pairs of complex phases to span every direction of $\RR^n$. In fact, let $\{e_k\}_{k=1}^n$ be the standard basis of $\RR^n$. There exist $2n$ real CGO solutions $w_k \in H^{s+1}(\Omega)$ for $1\leq k\leq 2n$, such that the $n$ vector fields given by $\nabla w_{2k-1} w_{2k} - w_{2k-1} \nabla w_{2k}$ are uniformly close in direction to $e_k$ in $H^s(\Omega)$ for all $1\leq k\leq n$ \cite[section 3.3]{Bal2010}. We summarize this observation in the following result:
    \begin{corollary}\label{cgo-3}
        There exists $2n$ real CGO solutions $w_k\in H^{s+1}(\Omega) \hookrightarrow W^{1,\infty}(\Omega)$ such that the matrix field given by $(\nabla w_{2k-1} w_{2k} - w_{2k-1} \nabla w_{2k})_{k=1}^n \in L^\infty(\Omega; \RR^{n\times n})$ is invertible with bounded inverse on $\Omega$.
    \end{corollary}
    \subsection{Sobolev Interpolations}
    In this section, we list some of the key Sobolev interpolation results needed in our proofs, and again assume $\Omega\subseteq \RR^n$ to be a bounded connected domain with smooth boundary.
    \begin{theorem}[Gagliardo--Nirenberg Inequalities {\cite{Nirenberg1959}, \cite[theorem 1.5.2]{Cherrier2012}}]
        Let $m \in \NN$, $1\leq p, r \leq \infty$, and $u \in L^p(\Omega) \cap L^r(\Omega)$. Assume that $D^m u \in L^p(\Omega)$. For an integer $0 \leq j \leq m$ and $j/m\leq \theta < 1$, define $q$ by
        \begin{align*}
            1/q = j/n + \theta(1/p - m/n) + (1-\theta)1/r
        \end{align*}
        Then, for any $\gamma\in \NN^n$, with $\vert\gamma\vert = j$, $D^\gamma u\in L^q(\Omega)$ and satisfies the Gagliardo--Nirenberg inequality
        \begin{align*}
            \Vert D^\gamma u\Vert_{L^q(\Omega)}\leq C\Vert D^m u\Vert_{L^p(\Omega)}^\theta \Vert u\Vert_{L^r(\Omega)}^{1-\theta} + C\Vert u\Vert_{L^s(\Omega)}
        \end{align*}
        with finite $1 \leq s \leq \max\{p, r\}$, and $C > 0$ is independent of $u$.
    \end{theorem}
    Applying the above inequality we obtain the following interpolation result:
    \begin{corollary}\label{sobolev-interpolation-1}
        Let $m\in \NN$ and assume that $u\in W^{m,\infty}(\Omega)$. Then the following interpolation inequality holds:
        \begin{align*}
            \Vert u\Vert_{L^\infty(\Omega)} \leq C\Vert u\Vert_{W^{m,\infty}(\Omega)}^{1-\theta} \Vert u\Vert_{L^2(\Omega)}^\theta
        \end{align*}
        where $\theta = 2m/(n + 2m)$ and $C>0$ is a constant independent of $u$.
    \end{corollary}
    \begin{proposition}[{\cite[proposition 1.5.3]{Cherrier2012}}]\label{sobolev-interpolation-2}
        Let $s_2 \geq s \geq s_1 \geq 0$, and $0 \leq \theta \leq 1$ be such that $s = (1 - \theta)s_1 + \theta s_2$. Then, for all $u \in H^{s_2}(\Omega)$, the inequality
        \begin{align*}
            \Vert u\Vert_{H^s(\Omega)} \leq C\Vert u\Vert_{H^{s_1}(\Omega)}^{1-\theta} \Vert u\Vert_{H^{s_2}(\Omega)}^\theta
        \end{align*}
        holds with $C > 0$ independent of $u$.
    \end{proposition}
    \subsection{Banach Inverse Function Theory}\label{app:banach}
    In this section, we provide a short proof of the nonlinear local inverse theorem given in \cite{Stefanov2009}, here written in \Cref{inverse-function-theorem}, which controls the nonlinear inverse stability of a function between Banach spaces in terms of the linearized map.
    \begin{proof}[Proof of~\Cref{{inverse-function-theorem}}]
        The claim follows directly from the following estimate. Given any $x\in U$ and $h\in X$ such that $x+h\in U$,
        \begin{align*}
            \Vert h\Vert_{X} &\leq C_1\Vert h \Vert_{X''}^{1-\mu_1}\Vert h\Vert_{X'}^{\mu_1} \leq C_1C_3^{\mu_1} \Vert h\Vert_{X''}^{1-\mu_1}\Vert\delta F_x(h)\Vert_{Y'}^{\mu_1}\\
            &\leq C_1^{1+\mu_1}C_3^{\mu_1} \Vert h\Vert_{X''}^{1-\mu_1}\Vert \delta F_x(h)\Vert_{Y''}^{\mu_1(1-\mu_2)}\Vert \delta F_x(h)\Vert_{Y}^{\mu_1\mu_2}\\
            &\leq \left(C_1^{1+\mu_1}C_2^{\mu_1(1-\mu_2)}C_3^{\mu_1} \Vert h\Vert_{X''}^{1-\mu_1\mu_2}\right)\Vert F(x + h) - F(x) - R_{x}(h)\Vert_{Y}^{\mu_1\mu_2},
        \end{align*}
        The constants can be bounded by restricting $h$ in $X''$. For the second term, we have
        \begin{align*}
            \Vert F(x + h) - F(x) - R_{x}(h)\Vert_{Y}^{\mu_1\mu_2} &\leq \Vert F(x + h) - F(x)\Vert_Y^{\mu_1\mu_2} + \Vert R_x(h)\Vert_Y^{\mu_1\mu_2}\\
            &\leq \Vert F(x + h) - F(x)\Vert_Y^{\mu_1\mu_2} + C_2^{\mu_1\mu_2} \Vert h\Vert_X^{2\mu_1\mu_2}
        \end{align*}
        Note that the last error term in $h$ is of degree $2\mu_1\mu_2 = 2\alpha > 1$. Subtracting it from the left-hand side and restricting to a small $X''$-neighborhood of $x_0$ proves the claim.
    \end{proof}
{
\bibliographystyle{alpha}
\bibliography{ref.bib}
}
\end{document}